\documentclass[twoside]{article} 
\usepackage[letterpaper, left=3.5cm, right=3.5cm, top=4cm, bottom=2.5cm]{geometry}
\usepackage{graphicx}
\usepackage{amsmath,amssymb,amsthm}
\usepackage{authblk}
\usepackage[utf8]{inputenc}
\usepackage{microtype}
\usepackage{mathrsfs} 
\usepackage{enumitem}
\usepackage{calc}
\usepackage{esint}
\usepackage{fancyhdr}
\usepackage{xcolor}
\usepackage[labelfont = bf]{caption}
\usepackage{doi}
\usepackage{subfigure}
\usepackage[numbers]{natbib}
\usepackage{float}
\usepackage{stmaryrd}
\usepackage{mathtools}
\usepackage{booktabs}
\usepackage{colortbl}
\usepackage{tabulary}
\usepackage{diagbox}
\usepackage{caption}
\usepackage{changepage}

\DeclareMathAlphabet{\mathbbold}{U}{bbold}{m}{n}
\newcommand{\mathbbone}{\mathbbold{1}}
\newcommand{\mathbbzero}{\mathbbold{0}}

\usepackage{physics} 
\usepackage{tikz}
\usepackage{mathdots}
\usepackage{yhmath}
\usepackage{cancel}
\usepackage{color}
\usepackage{siunitx}
\usepackage{array}
\usepackage{multirow}
\usepackage{amssymb}
\usepackage{gensymb}
\usepackage{tabularx}
\usepackage{extarrows}
\usepackage{booktabs}
\usetikzlibrary{fadings}
\usetikzlibrary{patterns}
\usetikzlibrary{shadows.blur}
\usetikzlibrary{shapes}

\newtheorem{theorem}{Theorem}[section]
\numberwithin{equation}{section}
\newtheorem{proposition}[theorem]{Proposition}

\newtheorem{remark}[theorem]{Remark}

\newtheorem{example}[theorem]{Example}

\newtheorem{algorithm}[theorem]{Algorithm}

\makeatletter
\renewcommand\@biblabel[1]{[#1]\hspace{0em}}
\makeatother

\usepackage{titlesec}
\titleformat{\section}{\normalfont\scshape\centering}{\thesection.}{0.5em}{}
\titleformat*{\subsection}{\itshape}
\titleformat*{\subsubsection}{\itshape}

\providecommand{\keywords}[1]
{
	{\small\emph{Keywords:} #1}
}

\providecommand{\MSC}[1]
{
	{\small\emph{AMS MSC (2020):~~} #1}
}

\definecolor{denim}{rgb}{0.08, 0.38, 0.74}
\definecolor{byzantium}{rgb}{0.44, 0.16, 0.39} 
\definecolor{shamrockgreen}{rgb}{0.0, 0.62, 0.38} 

\usepackage{hyperref}
\hypersetup{
	colorlinks=true,
	linkcolor=denim,
	citecolor = shamrockgreen,
	filecolor=magenta, 
	urlcolor=byzantium,
}

\begin{document}
	\setlength{\abovedisplayskip}{5.5pt}
	\setlength{\belowdisplayskip}{5.5pt}
	\setlength{\abovedisplayshortskip}{5.5pt}
	\setlength{\belowdisplayshortskip}{5.5pt}

	\title{\vspace*{-20mm}A $\operatorname{prox}$-Based Semi-Smooth Newton Method\\
for TV-Minimization} 
	\author[1]{Sören Bartels\thanks{Email: \url{bartels@mathematik.uni-freiburg.de}}}
	\author[2,3]{Alex Kaltenbach\thanks{Email: \url{kaltenbach@math.tu-berlin.de}}}
	\date{\today\vspace{-2.5mm}}
	\affil[1]{\small{Department of Applied Mathematics, University of Freiburg, Hermann--Herder-Str.\ 10,  D-79104 Freiburg}}\affil[2]{\small{Department of Mathematical Sciences and the Center for Mathematics and Artificial Intelligence (CMAI), George Mason University, Fairfax, VA 22030, USA.}}
	\affil[3]{\small{Institute of Mathematics, Technical University of Berlin, Stra\ss e des 17.\ Juni 136, D-10623 Berlin\vspace{-2.5mm}}}
	\maketitle

	\pagestyle{fancy}
	\fancyhf{}
	\fancyheadoffset{0cm}
	\addtolength{\headheight}{-0.25cm}
	\renewcommand{\headrulewidth}{0pt} 
	\renewcommand{\footrulewidth}{0pt}
	\fancyhead[CO]{\textsc{$\operatorname{prox}$-based semi-smooth Newton method for TV-minimization}}
	\fancyhead[CE]{\textsc{S. Bartels and A. Kaltenbach}}
	\fancyhead[R]{\thepage}
	\fancyfoot[R]{}
	
	\begin{abstract} 
        In this paper, we devise a $\operatorname{prox}$-based semi-smooth Newton method
for the non-differ\-entiable TV-minimization problem. To this end,
the  primal-dual optimality conditions~are reformulated
as a nonlinear operator equation with  Newton-(type-)differentiable~\mbox{structure}.
We investigate the question of well-posedness of the resulting semi-smooth Newton scheme
in the infinite-dimensional setting and identify structural properties of the
associated Newton-type derivatives. 
For a conforming finite element discretization, we prove that the resulting
semi-smooth Newton method is globally well-posed and locally super-linearly convergent.\linebreak
The approach extends to a large class of convex minimization problems, coincides with established semi-smooth Newton methods for obstacle problems,  satisfies~a~primal-dual~\mbox{invariance}, and, under suitable 
additional assumptions, is well-posed in the infinite-dimensional setting.\linebreak
Numerical experiments indicate a robust practical performance of the proposed method, including reliable reduction of the discrete primal-dual gap estimator to machine precision, robustness with respect to the choice of proximity parameters, an improved convergence basin compared to a canonical primal semi-smooth Newton method, and effective performance even for quadratically graded meshes using only a mesh-independent~initialization~criterion.
\end{abstract} 
    
	
	\MSC{49M15; 65K10; 65N30; 49M29; 49J52; 90C25; 94A08\vspace{-1mm}}
	
	\section{Introduction}\thispagestyle{empty}\vspace{-1mm}

\hspace*{5mm}Non-differentiable minimization problems arise in various applications such as image processing, fracture mechanics, sparse recovery, and more generally in nonlinear regression methods. Opposed to convex minimization problems with inequality constraints, the efficient iterative solution of this class of non-differentiable minimization problems is less understood. In this paper, we consider a non-differentiable problem that serves as a model problem in image denoising (\textit{cf}.~\cite{RudinOsherFatemi1992}) and, given a \emph{noisy image} $g \in L^{2}(\Omega)$ and a \emph{fidelity parameter} $\alpha>0$, seeks a \emph{denoised~image}~$u \in BV(\Omega)\cap L^{2}(\Omega)$ that minimizes the functional $I\colon BV(\Omega)\cap L^{2}(\Omega)\to \mathbb{R}$, for every $v\in BV(\Omega)\cap L^{2}(\Omega)$ defined by\enlargethispage{15mm}
\begin{align}\label{intro:primal}
  I(v) \coloneqq \vert \mathrm{D}v\vert (\Omega) + \frac{\alpha}{2}\int_{\Omega}{\vert v-g\vert ^2\,\mathrm{d}x}\,.
\end{align}
Here, $BV(\Omega)\coloneqq \{v\in L^1(\Omega)\mid \vert \mathrm{D}v\vert (\Omega)<+\infty\}$ is the \emph{space of functions with~bounded~\mbox{variation}}, where $\vert \mathrm{D}(\cdot)\vert (\Omega)\colon L^1_{\textup{loc}}(\Omega)\to \mathbb{R}\cup\{+\infty\}$, for every $v\in L^1_{\textup{loc}}(\Omega)$ defined by
\begin{align*}
    \vert \mathrm{D}v\vert (\Omega)\coloneqq \sup_{\phi\in (C_{\mathrm{c}}^{\infty}(\Omega))^d\,:\,\|\phi\|_{\infty,\Omega}\leq 1}{\bigg\{\int_{\Omega}{v\operatorname{div}\phi\,\mathrm{d}x}\bigg\}}\,,
\end{align*}
is the \emph{total variation} functional, which promotes piece-wise smooth reconstructions,~while~suppress\-ing noise, and
poses severe difficulties in the practical minimization~of~\eqref{intro:primal} due to its linear growth and non-differentiability  occurring at the origin. 
We address these difficulties by exploiting~the associated primal-dual structure to devise a semi-smooth Newton method.\newpage 

A (Fenchel) dual problem (in the sense of \cite[Rem.\ 4.2, p.\ 60/61]{EkelandTemam1999}) to the minimization~of~\eqref{intro:primal} was identified in \cite[Sec.\ 2]{HintermüllerKunisch2004} (see also \cite[Sec.\ 10.1.3, p.\ 305]{Bartels2015}) and 
    is given via the maximization of the dual energy functional $D\colon W^2_0(\operatorname{div};\Omega)\to \mathbb{R}\cup\{-\infty\}$, for every $y\in  W^2_0(\operatorname{div};\Omega)$  defined by
    \begin{align}\label{eq:rof_dual}
        D(y)\coloneqq  -I_{K_1(0)}^{\Omega}(y)-\frac{1}{2\alpha}\int_{\Omega}{\vert\!\operatorname{div} y+\alpha g\vert^2\,\mathrm{d}x}+\frac{\alpha}{2}\int_{\Omega}{\vert g\vert^2\,\mathrm{d}x}\,.
    \end{align}
    Here, $W^2_0(\operatorname{div};\Omega)\coloneqq \{y\in (L^2(\Omega))^d\mid \operatorname{div}y\in L^2(\Omega)\,,\;\langle y\cdot n,v\rangle_{\smash{W^{\smash{\frac{1}{2}},2}(\partial\Omega)}}=0\text{ for all } v\in \smash{W^{1,2}(\Omega)}\}$ and  
    the indicator functional $I_{K_1(0)}^{\Omega}\colon (L^2(\Omega))^d\to \mathbb{R}\cup\{+\infty\}$, for every $\widehat{y}\in (L^2(\Omega))^d$,~is~defined~by\vspace{-0.5mm}
    \begin{align*}
        I_{K_1(0)}^{\Omega}(\widehat{y})\coloneqq \begin{cases}
            0&\text{ if }\vert \widehat{y}\vert \leq 1\text{ a.e.\ in }\Omega\,,\\
            +\infty&\text{ else}\,.
        \end{cases}
    \end{align*}

By \cite[Prop.\ 2.1 \& Thm.\ 2.2]{HintermüllerKunisch2004} (see also \cite[Prop.\ 10.3]{Bartels2015}), there exists~a~(not~necessarily~unique) dual solution $z\in W^2_0(\operatorname{div};\Omega)$, \textit{i.e.}, a maximizer of \eqref{eq:rof_dual}, and a strong~\mbox{duality}~\mbox{relation}~\mbox{applies},~\textit{i.e.}, we have that $I(u)=D(z)$, which is equivalent to the convex optimality relations (\textit{cf}.~\cite[Prop.~10.4]{Bartels2015})\vspace{-4.5mm}
\begin{subequations}\label{eq:rof_optimality}
    \begin{alignat}{2}\label{eq:rof_optimality.1}
         \vert \mathrm{D}u\vert(\Omega)&=-\int_{\Omega}{u\operatorname{div} z\,\mathrm{d}x}\,,\\
        \operatorname{div} z&=\alpha (u-g)\quad \text{ a.e.\ in }\Omega\,.\label{eq:rof_optimality.2}
    \end{alignat}
    \end{subequations} 
    In the  case  $u\in W^{1,1}(\Omega)\cap L^2(\Omega)$, 
    by the equality condition in the Fenchel--Young inequality (\textit{cf}.\ \cite[Prop.~5.1, p.\ 21]{EkelandTemam1999}), the convex optimality relation \eqref{eq:rof_optimality.1} is equivalent to
    \begin{align}\label{eq:rof_optimality.2}
        z\in \left\{\begin{aligned}
            \,\smash{\bigl\{\tfrac{1}{\vert \nabla u\vert}}\nabla u\bigr\}&\quad\text{ if }\vert \nabla u\vert>0\,,\,\\
            \,K_1(0)&\quad\text{ if }\vert \nabla u\vert=0\,,\,
        \end{aligned}\right\}\quad \text{ a.e.\ in }\Omega\,.
    \end{align}  
    Inspired by \cite{GriesseLorenz2008}, which studies Tikhonov functionals with $\ell^1$-sparsity constraints, and \cite{ChambollePock2011}, which introduces a first-order primal-dual algorithm for TV-minimization, in this paper, using proximity operators, suitably regularized counterparts of the primal-dual optimality conditions in \eqref{eq:rof_optimality} are reformulated as nonlinear operator equations possessing Newton(-type)~\mbox{differentiability}~\mbox{properties}. This reformulation naturally motivates a semi-smooth Newton method, which is investigated both in the infinite- and finite-dimensional setting. The present work is closely~related~to~recent research on duality-based \textit{a priori} and \textit{a posteriori}~\mbox{error}~\mbox{control}~for~\mbox{convex}~\mbox{variational}~\mbox{problems}~(\textit{cf}.~\mbox{\cite{Bartels2021,BartelsKaltenbach2023}}).\enlargethispage{5.5mm}\vspace{-1.5mm}

\subsection{Related contributions}\vspace{-1mm}

\hspace{5mm}The existing literature on iterative methods for the approximation of the
non-differentiable TV-minimization problem is rich and can broadly be divided
into first- and second-order methods:
\begin{itemize}[noitemsep,topsep=2pt,leftmargin=!,labelwidth=\widthof{$\bullet$}]
    \item[$\bullet$] \emph{First-order methods:} Various first-order methods 
have been proposed and analyzed~in~the~literature, including, for instance,
gradient-descent and time-marching schemes (\textit{cf}.\ 
\cite{RudinOsherFatemi1992,VogelOman1996,MarquinaOsher2000,BartelsDieningNochetto2018}), 
proximal and accelerated gradient methods (\textit{cf}.\ 
\cite{AuslenderTeboulle2006,BeckTeboulle2009,BeckTeboulle2009b}),
dual projection and primal-dual hybrid gradient methods 
(\textit{cf}.\ \cite{Chambolle2004,ZhuChan2008,ChambollePock2011}),  
augmented Lagrangian methods (\textit{cf}.\ \cite{WuTai2010}), 
split Bregman methods (\textit{cf}.\ \cite{GoldsteinOsher2009}), 
and graph-cut-based relaxations 
(\textit{cf}.\ \cite{BoykovKolmogorov2001,DarbonSigelle2005}). These are typically rather slow with convergence behavior ${\mathcal O}(k^{-\smash{\frac{1}{2}}})$ or require a careful tuning~of~acceleration~parameters;

    \item[$\bullet$] \emph{Second-order methods:} Second-order methods include
second-order cone programming (SOCP) formulations
(\textit{cf}.\ \cite{GoldfarbYin2005}),
Newton and quasi-Newton methods (\textit{cf}.\ \cite{ChanZhouChan1995,VogelOman1996,DobsonVogel1997,TaiWintherZhangZheng2023}, including lagged-diffusivity fixed-point methods interpreted as quasi-Newton methods),
primal-dual semi-smooth Newton methods 
(\textit{cf}.\ \cite{ChanGolubMulet1999,HintermüllerStadler2006})
and dual semi-smooth Newton methods 
(\textit{cf}.~\cite{NgQiYangHuang2007,Souiai2010}).\vspace{2mm}
\end{itemize}

\emph{The outline of the article is as follows}. In Section \ref{sec:continuous}, we derive a $\operatorname{prox}$-based semi-smooth Newton method in the infinite-dimensional setting and investigate the question~of~its~\mbox{well-posedness}; first, for a regularized TV-minimization problem, then, for  general convex variational problems.\linebreak In Section \ref{sec:discrete}, we derive a $\operatorname{prox}$-based semi-smooth Newton method for a finite element~\mbox{discretization} of the regularized TV-minimization problem and establish its global well-posedness and local super-linear convergence. In Section \ref{sec:numerics}, we present numerical experiments supporting the~\mbox{theoretical} findings and illustrating the practical performance of the proposed method.
 
\newpage

\section{The infinite-dimensional setting}\label{sec:continuous}

\subsection{Regularized TV-minimization problem}\label{subsec:rof_continuous}

\hspace{5mm}We consider a canonical regularization of the energy functional \eqref{intro:primal}~that~leads~to~the~primal energy functional $I^{\varepsilon}\colon W^{1,1}(\Omega)\cap L^2(\Omega)\to \mathbb{R}$, for every $v\in W^{1,1}(\Omega)\cap L^2(\Omega)$~defined~by\vspace{-0.5mm}
\begin{align}\label{intro:primal_regular}
  I^{\varepsilon}(v) \coloneqq \int_{\Omega}{|\nabla v|_{\varepsilon}\,\mathrm{d}x} +\frac{\alpha}{2}\int_{\Omega}{\vert v-g\vert ^2\,\mathrm{d}x}\,,\\[-6mm]\notag
\end{align}
where $|\cdot|_{\varepsilon}\colon \smash{\mathbb{R}^d}\to [0,+\infty)$ denotes a regularization of the~\mbox{Euclidean}~length $|\cdot|\colon \mathbb{R}^d\to [0,+\infty)$~and $\varepsilon\hspace{-0.15em}>\hspace{-0.15em}0$ \hspace{-0.1mm}a \hspace{-0.1mm}regularization~\hspace{-0.1mm}parameter. \hspace{-0.1mm}Throughout \hspace{-0.1mm}the \hspace{-0.1mm}entire~\hspace{-0.1mm}paper,~\hspace{-0.1mm}as~\hspace{-0.1mm}\mbox{regularization}~\hspace{-0.1mm}${|\hspace{-0.1em}\cdot\hspace{-0.1em}|_{\varepsilon}\colon \hspace{-0.15em}\mathbb{R}^d\hspace{-0.15em}\to\hspace{-0.15em} [0,+\infty)}$, 
we employ the \emph{Huber regularization} (\textit{cf}.\ \cite{Huber1964}),~
for~every~$t\in \mathbb{R}^d$~defined~by\vspace{-0.5mm}
\begin{align*}
  |t|_{\varepsilon}
  \coloneqq
 \smash{\min\bigl\{
    |t| - \tfrac{\varepsilon}{2},
    \tfrac{|t|^{2}}{2\varepsilon}
  \bigr\}}\,,\\[-6mm]\notag
\end{align*}
which is obtained as the Moreau envelope (with regularization parameter $\varepsilon>0$) of the Euclidean length  $|\cdot|\colon \mathbb{R}^d\to [0,+\infty)$ (see, \textit{e.g.}, \cite[Expl.\ 6.54]{Beck2017}).

Due to a lack of a suitable notion of weak compactness of bounded sets in $W^{1,1}(\Omega)$,~in~general, the primal energy functional \eqref{intro:primal_regular} does not admit a minimizer. Therefore,~it~is~\mbox{common}~to~pass~to the \hspace{-0.1mm}BV-regularization \hspace{-0.1mm}(\textit{cf}.\ \hspace{-0.1mm}\cite[Sec.\ \hspace{-0.1mm}11.3]{AttouchButtazzoMichaille2014}) \hspace{-0.1mm}of \hspace{-0.1mm}the \hspace{-0.1mm}primal \hspace{-0.1mm}energy \hspace{-0.1mm}functional \hspace{-0.1mm}\eqref{intro:primal_regular}, \hspace{-0.1mm}\textit{i.e.}, \hspace{-0.1mm}the~\hspace{-0.1mm}\mbox{extended}~\hspace{-0.1mm}(not re-labelled) energy functional  
$I^{\varepsilon}\colon BV(\Omega)\cap L^2(\Omega)\to \mathbb{R}$, for every $v\in BV(\Omega)\cap L^2(\Omega)$~defined~by\vspace{-0.5mm}
\begin{align}\label{intro:primal_regular_extended}
  I^{\varepsilon}(v) \coloneqq \int_{\Omega}{|\nabla v|_{\varepsilon}\,\mathrm{d}x}+\vert \mathrm{D}^s v\vert (\Omega)+\frac{\alpha}{2}\int_{\Omega}{\vert v-g\vert ^2\,\mathrm{d}x}\,.\\[-6mm]\notag
\end{align}
Here, $\mathrm{D}^s v\hspace{-0.15em}\in\hspace{-0.15em} (\mathcal{M}(\Omega))^d$\footnote{$\mathcal{M}(\Omega)$ denotes the vector space of measures on $\Omega$.} denotes the singular part of the Lebesgue decomposition $\mathrm{D}v\hspace{-0.15em}=\hspace{-0.15em}\nabla v\otimes\mathrm{d}x+\mathrm{D}^s v$ in \hspace{-0.15mm}$(\mathcal{M}(\Omega))^d$ \hspace{-0.15mm}(\textit{cf}.\ \hspace{-0.15mm}\cite[\hspace{-0.25mm}Thm.\ \hspace{-0.25mm}10.4]{Bartels2015}), 
\hspace{-0.15mm}where \hspace{-0.15mm}$\nabla v\hspace{-0.175em}\coloneqq\hspace{-0.175em}\frac{\mathrm{D}^av}{\mathrm{d}x}\hspace{-0.175em}\in\hspace{-0.175em} (L^1(\Omega))^d$ \hspace{-0.15mm}is \hspace{-0.15mm}the \hspace{-0.15mm}Radon--Nikod\'ym~\hspace{-0.15mm}\mbox{derivative}~\hspace{-0.15mm}of~\hspace{-0.15mm}the absolutely continuous part $\mathrm{D}^av=\nabla
 v\otimes \mathrm{d}x\in (\mathcal{M}(\Omega))^d$ and  $\mathrm{D}^s v\perp \nabla v\otimes\mathrm{d}x$.  By the direct method in the calculus of variations 
 the functional \eqref{intro:primal_regular_extended} admits
 a unique minimizer $u^{\varepsilon}\in BV(\Omega)\cap L^2(\Omega)$. Moreover, due to \cite[Sec.\ 2]{SchreiberKaltenbach2026}, $I(u^{\varepsilon})\leq I^{\varepsilon}(u^{\varepsilon})+\tfrac{\varepsilon}{2}\vert \Omega\vert$,  and $I^{\varepsilon}(u^{\varepsilon})\leq I^{\varepsilon}(u)\leq I(u)$,
 we have that\vspace{-0.5mm}
\begin{align}\label{eq:regularization_error}
    \vert \mathrm{D}u^{\varepsilon}\vert (\Omega)+\int_{\Omega}{u^{\varepsilon}\operatorname{div} z\,\mathrm{d}x}+\frac{\alpha}{2}\int_{\Omega}{\vert u^{\varepsilon} -u\vert^2\,\mathrm{d}x}&= I(u^{\varepsilon})-I(u)
    \leq \tfrac{\varepsilon}{2}\vert \Omega\vert\,,\\[-6mm]\notag
\end{align}
which allows us to estimate the primal regularization error (since $\vert \mathrm{D}u^{\varepsilon}\vert (\Omega)+\smash{\int_{\Omega}{u^{\varepsilon}\operatorname{div} z\,\mathrm{d}x}}\ge 0$).\enlargethispage{7.5mm}

Analogously to \cite[Sec.\ 2]{HintermüllerKunisch2004}, one finds that 
a (Fenchel) dual problem (in the sense of \cite[Rem.~4.2, p.\ 60/61]{EkelandTemam1999}) to the minimization of \eqref{intro:primal_regular} is given via the maximization of the dual energy functional $D_{\varepsilon}\colon W^2_0(\operatorname{div};\Omega)\to \mathbb{R}\cup\{-\infty\}$, for every $y\in W^2_0(\operatorname{div};\Omega)$ defined by\vspace{-0.5mm}
\begin{align}\label{eq:rof_dual_regular}
    D_{\varepsilon}(y)\coloneqq -\frac{\varepsilon}{2}\int_{\Omega}{\vert y\vert^2\,\mathrm{d}x}-I_{K_1(0)}^{\Omega}(y)-\frac{1}{2\alpha}\int_{\Omega}{\vert\!\operatorname{div} y+\alpha g\vert^2\,\mathrm{d}x}+\frac{\alpha}{2}\int_{\Omega}{\vert g\vert^2\,\mathrm{d}x}\,.\\[-6mm]\notag
\end{align}

Moreover, by analogy with \cite[Prop.\ 2.1 \& Thm.\ 2.2]{HintermüllerKunisch2004}, one finds that there exists a unique dual solution $z^{\varepsilon}\in W^2_0(\operatorname{div};\Omega)$, \textit{i.e.}, a maximizer of \eqref{eq:rof_dual_regular}, and a strong duality relation applies, \textit{i.e.}, we have that $I^{\varepsilon}(u^{\varepsilon})=D^{\varepsilon}(z^{\varepsilon})$, which is equivalent to the optimality conditions~(\textit{cf}.~\mbox{\cite[Prop.~10.4]{Bartels2015}})\vspace{-0.5mm}
\begin{subequations}\label{eq:rof_optimality_reg}
    \begin{alignat}{2}\label{eq:rof_optimality_reg-1}
         \int_{\Omega}{\vert \nabla u^{\varepsilon}\vert_{\varepsilon}\,\mathrm{d}x}+\vert \mathrm{D}^su^{\varepsilon}\vert(\Omega)+\frac{\varepsilon}{2}\int_{\Omega}{\vert z^{\varepsilon}\vert^2\,\mathrm{d}x}&=-\int_{\Omega}{u^{\varepsilon}\operatorname{div} z^{\varepsilon}\,\mathrm{d}x}\,,\\
        \operatorname{div} z^{\varepsilon}&=\alpha (u^{\varepsilon}-g)\quad \text{ a.e.\ in }\Omega\,.\label{eq:rof_optimality_reg-2}
    \end{alignat}\\[-4.5mm]\notag
    \end{subequations} 
    If $u^{\varepsilon}\in W^{1,1}(\Omega)\cap L^2(\Omega)$, 
    by the equality condition in the Fenchel--Young inequality (\textit{cf}.\ \cite[Prop.~5.1, p.\ 21]{EkelandTemam1999}), the optimality condition~\eqref{eq:rof_optimality_reg-1}~is~\mbox{equivalent}~to\vspace{-0.5mm}
\begin{align}\label{intro:optimality_reg.2}
    z^{\varepsilon} = \mathrm{D}|\cdot|_{\varepsilon}(\nabla u^{\varepsilon})\quad \text{ a.e.\ in }\Omega\,.\\[-6mm]\notag
\end{align} 
Moreover, due to \cite[Sec.\ 2]{SchreiberKaltenbach2026} and $D(z^{\varepsilon})\geq D^{\varepsilon}(z^{\varepsilon})\geq D^{\varepsilon}(z)\geq D(z)-\frac{\varepsilon}{2}\vert \Omega\vert$,
 we have that\vspace{-0.5mm}
\begin{align*}
    \vert \mathrm{D}u\vert (\Omega)+\int_{\Omega}{u\operatorname{div} z_\varepsilon\,\mathrm{d}x}+\frac{1}{2\alpha}\int_{\Omega}{\vert\!\operatorname{div}z_\varepsilon -\operatorname{div}z\vert^2\,\mathrm{d}x}= -D(z^{\varepsilon})+D(z)
    \leq \tfrac{\varepsilon}{2}\vert \Omega\vert\,,\\[-6mm]\notag
\end{align*}
which allows us to estimate the dual regularization error (since $\vert \mathrm{D}u\vert (\Omega)+\smash{\int_{\Omega}{u\operatorname{div} z_\varepsilon\,\mathrm{d}x}}\ge 0$). 


We note that the regularization via the Moreau envelope in the primal energy functional \eqref{intro:primal_regular} is essential for the semi-smooth Newton method considered below:
in the infinite-dimensional~setting,\linebreak it
ensures at least injectivity and a dense range of the Newton-type derivative of the modified optimality conditions derived below; in the finite-dimensional setting, it even ensures invertibility with uniformly bounded inverse of the Newton derivative  of the modified discrete optimality conditions derived later.
These modified optimality equations derived below are given~via~a~nonlinear operator that is  Lipschitz continuous uniformly with respect to the~regularization~\mbox{parameter}~$\varepsilon>0$, which is in strong contrast to the original optimality conditions \eqref{eq:rof_optimality_reg} (or the optimality condition \eqref{intro:optimality_reg.2} if $u^{\varepsilon}\in W^{1,1}(\Omega)\cap L^2(\Omega)$, respectively) for the primal energy functional \eqref{intro:primal_regular}.

Next, for the rest of this section, let us assume that $u^{\varepsilon}\hspace{-0.1em}\in\hspace{-0.1em} W^{1,1}(\Omega)\cap L^2(\Omega)$, so~that~the~\mbox{optimality} condition \eqref{intro:optimality_reg.2} applies. Then,
in order to reduce the effect of the discontinuity of the derivative of the Euclidean length $\vert\cdot\vert_{\varepsilon}\colon \mathbb{R}^d\to [0,+\infty)$ at the origin, we choose an arbitrary proximity parameter $\gamma >0$ and rewrite the optimality condition~\eqref{intro:optimality_reg.2}~as
\begin{align}\label{intro:optimality_reg.3}
  (\operatorname{id}_{\mathbb{R}^d} + \gamma \mathrm{D}|\cdot|_{\varepsilon})( \nabla u^{\varepsilon} )
  = \nabla u^{\varepsilon} + \gamma z^{\varepsilon}\quad\text{ a.e.\ in }\Omega\,.
\end{align}
Inasmuch as $\operatorname{id}_{\mathbb{R}^d} + \gamma \mathrm{D}|\cdot|_{\varepsilon}\colon \mathbb{R}^d\to \mathbb{R}^d$ is $1$-strongly monotone\footnote{For a normed vector space $(X,\|\cdot\|_X)$, an operator $A\colon X\to X^*$ is called \emph{$\mu$-strongly monotone}, where $\mu>0$, if $\langle Ax-Ay,x-y\rangle_X\ge \mu\|x-y\|_X^2$.} and continuous, the main theorem on monotone operators (\textit{cf}.\ \cite[Thm.\ 26.A]{Zeidler1990IIB}) implies the existence of a $1$-Lipschitz\footnote{For normed vector spaces $(X,\|\cdot\|_X)$ and $(Y,\|\cdot\|_Y)$, an operator $A\colon X\to Y$ is called \emph{$\mu$-Lipschitz continuous}, where $\mu>0$, if $\| Ax-Ay\|_Y\le \mu\|x-y\|_X$.} continuous and strictly monotone inverse $\operatorname{prox}_{\gamma \vert \cdot\vert_{\varepsilon}}\hspace{-0.1em}\coloneqq\hspace{-0.1em} (\operatorname{id}_{\mathbb{R}^d} + \gamma \mathrm{D}|\cdot|_{\varepsilon})^{-1}\colon \hspace{-0.1em}\mathbb{R}^d\hspace{-0.1em}\to\hspace{-0.1em} \mathbb{R}^d$, called~\emph{\mbox{proximity}~\mbox{operator}}~of~${\gamma\vert \hspace{-0.1em}\cdot\hspace{-0.1em}\vert_{\varepsilon}}$\linebreak (or \emph{resolvent operator} of $\gamma\mathrm{D}|\cdot|_{\varepsilon}$) (see,~\textit{e.g.},~\cite{BauschkeCombettes2017,Beck2017}, for more details~about~\mbox{proximity}~\mbox{operators}), which, for every $t\in \mathbb{R}^d$, is given via 
\begin{align}\label{def:prox}
    \operatorname{prox}_{\gamma\vert \cdot\vert_{\varepsilon}}(t)=\max\bigl\{\tfrac{\varepsilon}{\varepsilon+\gamma},1-\tfrac{\gamma}{\vert t\vert}\bigr\}t\,.
\end{align}
By means of the proximity operator \eqref{def:prox}, we may rewrite \eqref{intro:optimality_reg.3} as
\begin{align}\label{intro:optimality_reg.4}
    \nabla u^{\varepsilon}=\operatorname{prox}_{\gamma\vert \cdot\vert_{\varepsilon}}(\nabla u^{\varepsilon} + \gamma z^{\varepsilon})\quad\text{ a.e.\ in }\Omega\,.
\end{align}
As a consequence, introducing the abbreviated notation for the product spaces
\begin{align*}
    \mathbb{A}\coloneqq 
W^2_0(\operatorname{div};\Omega)
\times
(W^{1,1}(\Omega)\cap L^2(\Omega))\,,
\qquad
\mathbb{B}
\coloneqq
(L^1(\Omega))^d\times L^2(\Omega)\,,
\end{align*} 
the optimality conditions  \eqref{eq:rof_optimality_reg} are equivalent to the root finding problem for the nonlinear mapping $\mathtt{F}_{\varepsilon}\colon \mathbb{A} \to \mathbb{B}$, for every $(y,v) \in \mathbb{A}$ defined by
\begin{align}\label{definition:F_eps}
    \mathtt{F}_{\varepsilon}(y,v) \coloneqq
  \begin{bmatrix}
    \nabla v - \operatorname{prox}_{\gamma\vert \cdot\vert_{\varepsilon}}(\nabla v + \gamma y)\\
    \alpha(v-g) - \operatorname{div} y
  \end{bmatrix}\quad \text{ in }\mathbb{B}\,,
\end{align}
which, due to \eqref{intro:optimality_reg.4} and \eqref{eq:rof_optimality_reg-2}, admits the root $(z^{\varepsilon},u^{\varepsilon})\in \mathbb{A}$, \textit{i.e.},  we have that
\begin{align*}
   \mathtt{F}_{\varepsilon}(z^{\varepsilon},u^{\varepsilon})=(\mathtt{0}_d,0)\quad\text{ in }\mathbb{B}\,.
\end{align*}
Note that $\operatorname{prox}_{\gamma\vert \cdot\vert_{\varepsilon}}\colon \mathbb{R}^d\to \mathbb{R}^d$ is both $1$-Lipschitz continuous and monotone (\textit{cf}.\ \cite[Prop.\ 12.27]{BauschkeCombettes2017}). 
The Lipschitz continuity, by Rademacher's theorem~(\textit{cf}.~\cite{Rademacher1920}), implies that $\operatorname{prox}_{\gamma\vert \cdot\vert_{\varepsilon}}\colon \hspace{-0.1em}\mathbb{R}^d\hspace{-0.1em}\to\hspace{-0.1em} \mathbb{R}^d$~is~a.e.\ differentiable \hspace{-0.1mm}and, \hspace{-0.1mm}by \hspace{-0.1mm}\cite[Thm.\ \hspace{-0.1mm}2.6]{ChenNashedQi2000}, \hspace{-0.1mm}that \hspace{-0.1mm}it \hspace{-0.1mm}is \hspace{-0.1mm}point-wise \hspace{-0.1mm}Newton \hspace{-0.1mm}differentiable,~\hspace{-0.1mm}\textit{i.e.},~\hspace{-0.1mm}for~\hspace{-0.1mm}\mbox{every}~\hspace{-0.1mm}${t\hspace{-0.1em}\in\hspace{-0.1em} \mathbb{R}^d}$, there exists some $\delta_t>0$ and a Newton derivative $\mathtt{J}_{\smash{\operatorname{prox}_{\gamma\vert \cdot\vert_{\varepsilon}}}}\colon B_{\delta_t}^d(t)\to \mathbb{R}^{d\times d}$ such that
\begin{align*}
    \lim_{ h\to 0}{\bigg\{\frac{\vert \operatorname{prox}_{\gamma\vert \cdot\vert_{\varepsilon}}(t+h)-\operatorname{prox}_{\gamma\vert \cdot\vert_{\varepsilon}}(t)-\mathtt{J}_{\operatorname{prox}_{\gamma\vert \cdot\vert_{\varepsilon}}}(t+h)h\vert}{\vert h\vert }\bigg\}}=0\,.
\end{align*} 
Moreover, by the $1$-Lipschitz continuity and monotonicity  of $\operatorname{prox}_{\gamma\vert \cdot\vert_{\varepsilon}}\colon \mathbb{R}^d\to \mathbb{R}^d$, 
we have that\enlargethispage{6mm}
\begin{align}\label{eq:J_prox_lower_upper}
    \mathbbzero \preceq \mathtt{J}_{\smash{\operatorname{prox}_{\gamma\vert \cdot\vert_{\varepsilon}}}}(t)\preceq \mathbbone\,.
\end{align}
Here, by $\mathbbzero,\mathbbone\coloneqq(\delta_{ij})_{i,j\in \{1,\ldots,d\}}\in \mathbb{R}^{d\times d}$ we denote the zero and the identity matrix, respectively, and for 
$\mathtt{A},\mathtt{B}\in \mathbb{R}^{d\times d}$, we write $\mathtt{A}\preceq \mathtt{B}$ (or $\mathtt{A}\succeq \mathtt{B}$) if 
$((\mathtt{B}-\mathtt{A})t)\cdot t\ge 0$ (or $((\mathtt{B}-\mathtt{A})t)\cdot t\le 0$)~for~all~$t\in \mathbb{R}^d$.\pagebreak

Note that the Newton derivative $\mathtt{J}_{\smash{\operatorname{prox}_{\gamma\vert \cdot\vert_{\varepsilon}}}}\colon \hspace{-0.1em}B_{\delta_t}^d(t)\hspace{-0.1em}\to\hspace{-0.1em} \mathbb{R}^{d\times d}$, in general, depends on $t\hspace{-0.1em}\in\hspace{-0.1em} \mathbb{R}^d$~and~is not unique. In fact, for any $\lambda\in [0,1]$, a possible (global) Newton derivative $\mathtt{J}_{\smash{\operatorname{prox}_{\gamma\vert \cdot\vert_{\varepsilon}}}}\colon \mathbb{R}^d\to \mathbb{R}^{d\times d}$, for every $t\in \mathbb{R}^d$, is given via
\begin{align}\label{def:prox_newton}
  \mathtt{J}_{\operatorname{prox}_{\gamma\vert \cdot\vert_{\varepsilon}}}(t) =
  \begin{cases}
    \mathbbone - \frac{\gamma}{|t|}\,\mathbb{P}_{t} &\text{ if } |t| > \gamma + \varepsilon\,,\\
    (1-\lambda)\frac{\varepsilon}{\varepsilon + \gamma}\mathbbone+\lambda(\mathbbone - \frac{\gamma}{|t|}\,\mathbb{P}_{t}) &\text{ if } |t| = \gamma + \varepsilon \,,\\
    \frac{\varepsilon}{\varepsilon + \gamma}\mathbbone &\text{ if } |t| < \gamma + \varepsilon\,,
  \end{cases}
\end{align}
where $\mathbb{P}_{t} \coloneqq \mathbbone - \smash{\frac{t\otimes t}{|t|^2}}\in \mathbb{R}^{d\times d}$ denotes the orthogonal projection from $\mathbb{R}^d$ into $(\mathbb{R}t)^\perp$.

For the rest of the paper, we reserve the notation $\mathtt{J}_{\operatorname{prox}_{\gamma\vert \cdot\vert_{\varepsilon}}}\colon \mathbb{R}^d\to \mathbb{R}^{d\times d}$ for the (global)~Newton derivative defined by \eqref{def:prox_newton} for the particular choice  $\lambda=0$; for which the lower bound in \eqref{eq:J_prox_lower_upper}, for every $t\in \mathbb{R}^d$, improves to
\begin{align}\label{eq:J_prox_lower_upper_improved}
    \mathtt{J}_{\operatorname{prox}_{\gamma\vert \cdot\vert_{\varepsilon}}}(t)\succeq \tfrac{\varepsilon}{\varepsilon+\gamma}\mathbbone\,.
\end{align}

The \hspace{-0.1mm}crucial \hspace{-0.1mm}feature \hspace{-0.1mm}of \hspace{-0.1mm}representation \hspace{-0.1mm}\eqref{intro:optimality_reg.4} \hspace{-0.1mm}of \hspace{-0.1mm}the \hspace{-0.1mm}optimality \hspace{-0.1mm}condition \hspace{-0.1mm}\eqref{intro:optimality_reg.2} \hspace{-0.1mm}is \hspace{-0.1mm}that \hspace{-0.1mm}the \hspace{-0.1mm}proximity operator $\operatorname{prox}_{\gamma\vert \cdot\vert_{\varepsilon}}\colon \mathbb{R}^d\to \mathbb{R}^d$ is Lipschitz continuous with a constant~that~is~\mbox{independent}~of~${\varepsilon>0}$,
which \hspace{-0.1mm}implies \hspace{-0.1mm}that \hspace{-0.1mm}its \hspace{-0.1mm}Newton \hspace{-0.1mm}derivatives \hspace{-0.1mm}are \hspace{-0.1mm}bounded \hspace{-0.1mm}independently \hspace{-0.1mm}of  \hspace{-0.1mm}$\varepsilon\hspace{-0.175em}>\hspace{-0.175em}0$~\hspace{-0.1mm}(\textit{cf}.~\hspace{-0.1mm}\eqref{eq:J_prox_lower_upper}).~\hspace{-0.1mm}As~\hspace{-0.1mm}a~\hspace{-0.1mm}\mbox{result}, the nonlinear mapping $\mathtt{F}_{\varepsilon}\colon \mathbb{A} \to \mathbb{B}$, defined~by~\eqref{definition:F_eps}, 
is \hspace{-0.1mm}Lipschitz \hspace{-0.1mm}continuous \hspace{-0.1mm}with \hspace{-0.1mm}a \hspace{-0.1mm}constant \hspace{-0.1mm}that \hspace{-0.1mm}is \hspace{-0.1mm}independent~\hspace{-0.1mm}of~\hspace{-0.1mm}${\varepsilon\hspace{-0.15em}>\hspace{-0.15em}0}$, \hspace{-0.1mm}which \hspace{-0.1mm}by~\hspace{-0.1mm}\mbox{\cite[Thm.~\hspace{-0.1mm}2.6]{ChenNashedQi2000}},~\mbox{implies} that it is point-wise~\mbox{Newton}~\mbox{differentiable}, \textit{i.e.}, for every $(y,v)\hspace{-0.1em}\in \hspace{-0.1em}\mathbb{A}$, there exist  $\delta_{(y,v)}\hspace{-0.1em}>\hspace{-0.1em}0$ and a Newton~derivative $\mathtt{J}_{\mathtt{F}_{\varepsilon}}\colon B_{\smash{\delta_{(y,v)}}}^{\mathbb{A}}(y,v)\hspace{-0.1em}\to\hspace{-0.1em}\mathcal{L}(\mathbb{A};\mathbb{B})$\footnote{For normed vector spaces $(X,\|\cdot\|_X)$ and $(Y,\|\cdot\|_Y)$, by $\mathcal{L}(X;Y)$ we denote the \emph{space of linear and bounded operators} mapping from $X$ into $Y$.} such that
\begin{align}\label{eq:F_eps_Newton_diff}
    \lim_{(\tau_1,\tau_2)\to 0}{\bigg\{\frac{\| \mathtt{F}_{\varepsilon}(y+\tau_1,v+\tau_2)-\mathtt{F}_{\varepsilon}(y,v)-\mathtt{J}_{\mathtt{F}_{\varepsilon}}(y,v)(\tau_1,\tau_2)\|_{\mathbb{B}}}{\| (\tau_1,\tau_2)\|_{\mathbb{A}} }\bigg\}}=0\,.
\end{align} 
The difficulty is that this point-wise Newton derivative, in general, depends
on the evaluation~point.\linebreak At a given root
$(z^{\varepsilon},u^{\varepsilon}) \in \mathbb{A}$, this dependence renders it
impractical for use in a semi-smooth Newton method, as the respective root is
not known \emph{a priori}. For this reason, it is indispensable to identify  a  Newton
 derivative that can be evaluated  without prior knowledge of a root.
However, constructing an explicit instance of such a derivative is
severely obstructed~by~the~typical~\emph{norm~gap phenomenon}
(\textit{cf}.\ \cite[Prop.~4.1]{HintermüllerItoKunisch2003}),
which likewise occurs in the present setting.

Certainly, a natural candidate for a Newton derivative is given via the \textit{`Newton-type~derivative'} $\mathtt{J}_{\mathtt{F}_{\varepsilon}}\colon \mathbb{A}\to \mathcal{L}(\mathbb{A};\mathbb{B})$, for every $(y,v),(\widehat{y},\widehat{v})\in \mathbb{A} $ defined by
\begin{align}\label{definition:J_F_eps}
    \mathtt{J}_{\mathtt{F}_{\varepsilon}}(y,v)(\widehat{y},\widehat{v}) \coloneqq
  \begin{bmatrix}
    ( \mathbbone - \mathtt{J}_{\operatorname{prox}_{\gamma\vert \cdot\vert_{\varepsilon}}}(a))\nabla \widehat{v} - \gamma\mathtt{J}_{\operatorname{prox}_{\gamma\vert \cdot\vert_{\varepsilon}}}(a)\widehat{y}\\
    \alpha \widehat{v} - \operatorname{div} \widehat{y}
  \end{bmatrix}\quad\text{ in }\mathbb{B}\,,
\end{align} 
where $a \hspace{-0.12em}\coloneqq \hspace{-0.12em}\nabla v + \gamma y \hspace{-0.12em}\in \hspace{-0.12em} (L^1(\Omega))^d$, which may still represent a Newton derivative~in~a~\mbox{generalized}~sense, \textit{i.e.}, for the restricted mapping ${\mathtt{F}_{\varepsilon}\colon W^{2+}_0(\operatorname{div};\Omega)\times (W^{1,1+}(\Omega)\cap L^{2+}(\Omega)) \to \mathbb{B}}$, defined~by~\eqref{definition:F_eps}; which is why we employ the same notation as for a genuine Newton derivative satisfying Newton differentiability condition \eqref{eq:F_eps_Newton_diff}.

Therefore, we consider the Newton-type derivative
\eqref{definition:J_F_eps} as a candidate for a variation that avoids the dependence on a root, but, due to the norm gap phenomenon,
satisfies the Newton differentiability condition \eqref{eq:F_eps_Newton_diff} only in more
restrictive function~spaces~rather~than~the~\mbox{natural}~ones. 

Next, we investigate structural features of the Newton-type derivative \eqref{definition:J_F_eps} that yield fundamental insights into its invertibility behavior in the context of semi-smooth~Newton~methods.

\begin{proposition}\label{prop:J_F_eps}   
    For every $(y,v)\in \mathbb{A}$, the Newton-type derivative $\mathtt{J}_{\mathtt{F}}(y,v)\in \mathcal{L}(\mathbb{A};\mathbb{B})$, defined by \eqref{definition:J_F_eps}, is injective and has a dense range.
\end{proposition}\pagebreak

\begin{proof} Let $(y,v)\in \mathbb{A}$ be fixed, but arbitrary.

    \emph{1.\ Injectivity:} Let $(\widehat{y},\widehat{v})\hspace{-0.15em}\in \hspace{-0.15em}\operatorname{ker}\mathtt{J}_{\mathtt{F}_{\varepsilon}}(y,v)$ be fixed, but arbitrary.  
    Then,~we~have~that~$\mathtt{J}_{\mathtt{F}_{\varepsilon}}(y,v)(\widehat{y},\widehat{v}) $ $=(\mathtt{0}_d,0)$ a.e.\ in $\Omega$, which equivalently reads
\begin{subequations}\label{prop:J_F_eps.1}   
\begin{alignat}{2}\label{prop:J_F_eps.1.1}   
  \tfrac{\gamma}{|a|}\mathbb{P}_{a}\nabla \widehat{v} -
  \gamma( \mathbbone - \tfrac{\gamma}{|a|}\mathbb{P}_{a})\widehat{y} &= 0
  &&\quad\text{ in }A \coloneqq \{ |a| > \gamma + \varepsilon\}\,,\\\label{prop:J_F_eps.1.2}   
  \tfrac{\gamma}{\varepsilon + \gamma}\nabla \widehat{v} -
  \tfrac{\gamma\varepsilon}{\varepsilon + \gamma}\widehat{y} &= 0
  &&\quad\text{ in }A^{\smash{\complement}}\,,\\
  \operatorname{div}\widehat{y}&=\alpha\widehat{v}&&\quad\text{ in }\Omega\,.\label{prop:J_F_eps.1.3}   
\end{alignat}
\end{subequations}
From \eqref{prop:J_F_eps.1.1}, we infer that $\mathbb{P}_{a}^{\perp}\widehat{y} = \mathtt{0}_d$ a.e.\ in $A$  and, consequently,  
$\mathbb{P}_{a}\nabla \widehat{v} = (|a|-\gamma) \mathbb{P}_{a} \widehat{y}$~a.e.~in~$A$, while \eqref{prop:J_F_eps.1.2}  reduces to 
$\nabla \widehat{v} = \varepsilon \widehat{y}$ a.e.\ in $A^{\smash{\complement}}$. Then, using the latter, that $\mathbb{P}_{a} \widehat{y}=\widehat{y}$ a.e.\ in $A$, that $\vert a\vert-\gamma>\varepsilon$ a.e.\ in $A$, integration-by-parts, and  \eqref{prop:J_F_eps.1.3}, we find that\enlargethispage{2.5mm}
\begin{align*}
    \varepsilon\int_{\Omega}{\vert \widehat{y}\vert^2\,\mathrm{d}x}
  &\le \varepsilon\int_{A^{\smash{\complement}}}{\widehat{y}\cdot\widehat{y}\,\mathrm{d}x} +\int_{A}{(|a|-\gamma)\mathbb{P}_{a}\widehat{y}\cdot\widehat{y}\,\mathrm{d}x} \\ 
  &= \int_{A^{\smash{\complement}}}{\nabla \widehat{v}\cdot\widehat{y}\,\mathrm{d}x} +\int_{A}{\mathbb{P}_{a}\nabla\widehat{v}\cdot\widehat{y}\,\mathrm{d}x} 
  \\&=\int_{\Omega}{\nabla \widehat{v}\cdot\widehat{y}\,\mathrm{d}x}
   \\&=-\int_{\Omega}{\widehat{v}\operatorname{div}\widehat{y}\,\mathrm{d}x}
  \\&= -\alpha \int_{\Omega}{\vert \widehat{v}\vert^2\,\mathrm{d}x} \le 0\,,
\end{align*}
\textit{i.e.}, \hspace{-0.1mm}$\widehat{y}\hspace{-0.15em}=\hspace{-0.15em}\mathtt{0}_d$ \hspace{-0.1mm}a.e.\ \hspace{-0.1mm}in \hspace{-0.1mm}$\Omega$ \hspace{-0.1mm}and, \hspace{-0.1mm}thus, \hspace{-0.1mm}$\widehat{v}\hspace{-0.15em}=\hspace{-0.15em}0$ \hspace{-0.1mm}a.e.\ \hspace{-0.1mm}in \hspace{-0.1mm}$\Omega$. \hspace{-0.1mm}In \hspace{-0.1mm}summary, \hspace{-0.1mm}we \hspace{-0.1mm}conclude \hspace{-0.1mm}that \hspace{-0.1mm}${\operatorname{ker}\mathtt{J}_{\mathtt{F}_{\varepsilon}}(y,v)\hspace{-0.15em}=\hspace{-0.15em}\{(\mathtt{0}_d,0)\}}$, \textit{i.e.}, the claimed injectivity of $\mathtt{J}_{\mathtt{F}}(y,v)\in \mathcal{L}(\mathbb{A};\mathbb{B})$.

 \hspace{-1mm}\emph{2.\ \hspace{-0.1mm}Dense \hspace{-0.1mm}range:} \hspace{-0.1mm}It \hspace{-0.1mm}suffices \hspace{-0.1mm}to \hspace{-0.1mm}show \hspace{-0.1mm}that \hspace{-0.1mm}the \hspace{-0.1mm}adjoint \hspace{-0.1mm}operator \hspace{-0.1mm}$\mathtt{J}_{\mathtt{F}_{\varepsilon}}^*(y,v)\hspace*{-0.1em}\coloneqq\hspace*{-0.1em}(\mathtt{J}_{\mathtt{F}_{\varepsilon}}(y,v))^*\hspace*{-0.1em}\in\hspace*{-0.1em} \mathcal{L}(\mathbb{B}^*; \mathbb{A}^*)$ is injective (\textit{cf}.\ \cite[Cor.\ 2.18(ii)]{Brezis2011}). 
 To this end, let ${(\widehat{y},\widehat{v})\in  \operatorname{ker}\mathtt{J}_{\mathtt{F}}^*(y,v)}$ be fixed, but arbitrary. Then, for every ${(\widetilde{y},\widetilde{v})\in \mathbb{A}}$, we~have~that
 \begin{align}\label{prop:J_F_eps.2}
    \begin{aligned}
     (\mathtt{0}_d,0) &=\langle \mathtt{J}_{\mathtt{F}_{\varepsilon}}^*(y,v)(\widehat{y},\widehat{v}),(\widetilde{y},\widetilde{v})\rangle_{\mathbb{A}}\\&=\langle (\widehat{y},\widehat{v}),\mathtt{J}_{\mathtt{F}_{\varepsilon}}(y,v)(\widetilde{y},\widetilde{v})\rangle_{\mathbb{B}}
     \\&=\int_{\Omega}{\widehat{y}\cdot\bigl((\mathbbone-\mathtt{J}_{\operatorname{prox}_{\gamma\vert \cdot\vert_{\varepsilon}}}(a))\nabla \widetilde{v}-\gamma\mathtt{J}_{\operatorname{prox}_{\gamma\vert \cdot\vert_{\varepsilon}}}(a)\widetilde{y}\bigr)\,\mathrm{d}x}
     +\int_{\Omega}{\widehat{v}(\alpha \widetilde{v}-\operatorname{div}\widetilde{y})\,\mathrm{d}x}\,. 
     \end{aligned}
 \end{align}
    Choosing $(\widetilde{y},0)\in \mathbb{A}$ in \eqref{prop:J_F_eps.2} and using that ${\mathtt{J}_{\operatorname{prox}_{\gamma\vert \cdot\vert_{\varepsilon}}}(a)^\top=\mathtt{J}_{\operatorname{prox}_{\gamma\vert \cdot\vert_{\varepsilon}}}(a)}$ a.e.\ in $\Omega$, we infer that
 \begin{align*}
     \int_{\Omega}{\gamma\mathtt{J}_{\operatorname{prox}_{\gamma\vert \cdot\vert_{\varepsilon}}}(a)\widehat{y}\cdot\widetilde{y}\,\mathrm{d}x}=-\int_{\Omega}{\widehat{v}\operatorname{div}\widetilde{y}\,\mathrm{d}x}\,,
 \end{align*}
 which implies that $\widehat{v}\in W^{1,1}(\Omega)\cap L^2(\Omega)$ with
 \begin{align}\label{prop:J_F_eps.3}
     \nabla \widehat{v}=\gamma\mathtt{J}_{\operatorname{prox}_{\gamma\vert \cdot\vert_{\varepsilon}}}(a)\widehat{y}\quad\text{ a.e.\ in }\Omega\,.
 \end{align}
 Then, choosing $(\widetilde{y},\widetilde{v})=(\mathtt{0}_d,\widehat{v})\in \mathbb{A}$ in \eqref{prop:J_F_eps.2}, using \eqref{prop:J_F_eps.3}~and~that, due to \eqref{eq:J_prox_lower_upper}, there holds $(\mathbbone-\mathtt{J}_{\operatorname{prox}_{\gamma\vert \cdot\vert_{\varepsilon}}}(a))\mathtt{J}_{\operatorname{prox}_{\gamma\vert \cdot\vert_{\varepsilon}}}(a)\succeq 0$ a.e.\ in $\Omega$, we find that
 \begin{align*}
    0&=\int_{\Omega}{\widehat{y}\cdot(\mathbbone-\mathtt{J}_{\operatorname{prox}_{\gamma\vert \cdot\vert_{\varepsilon}}}(a))\nabla \widehat{v}\,\mathrm{d}x}+\alpha \int_{\Omega}{\vert \widehat{v}\vert^2\,\mathrm{d}x}
    \\&=\gamma\int_{\Omega}{\widehat{y}\cdot(\mathbbone-\mathtt{J}_{\operatorname{prox}_{\gamma\vert \cdot\vert_{\varepsilon}}}(a))\mathtt{J}_{\operatorname{prox}_{\gamma\vert \cdot\vert_{\varepsilon}}}(a)\widehat{y}\,\mathrm{d}x}+\alpha \int_{\Omega}{\vert \widehat{v}\vert^2\,\mathrm{d}x}
    \\&\ge \alpha \int_{\Omega}{\vert \widehat{v}\vert^2\,\mathrm{d}x}\,,
 \end{align*}
 \textit{i.e.}, \hspace{-0.1mm}$\widehat{v}\hspace{-0.15em}=\hspace{-0.15em}0$ \hspace{-0.1mm}a.e.\ \hspace{-0.1mm}in \hspace{-0.1mm}$\Omega$ \hspace{-0.1mm}and, \hspace{-0.1mm}thus, \hspace{-0.1mm}$\widehat{y}\hspace{-0.15em}=\hspace{-0.15em}\mathtt{0}_d$ \hspace{-0.1mm}a.e.\ \hspace{-0.1mm}in \hspace{-0.1mm}$\Omega$. \hspace{-0.1mm}In \hspace{-0.1mm}summary, \hspace{-0.1mm}we conclude~\hspace{-0.1mm}that~\hspace{-0.1mm}${\operatorname{ker}\mathtt{J}_{\mathtt{F}_{\varepsilon}}^*(y,v)\hspace{-0.15em}=\hspace{-0.15em}\{(\mathtt{0}_d,0)\}}$, \textit{i.e.}, \hspace{-0.1mm}the \hspace{-0.1mm}claimed \hspace{-0.1mm}injectivity \hspace{-0.1mm}of \hspace{-0.1mm}${\mathtt{J}_{\mathtt{F}_{\varepsilon}}^*(y,v)\hspace{-0.15em}\in\hspace{-0.15em} \mathcal{L}(\mathbb{B}^*; \mathbb{A}^*)}$.
\end{proof}\pagebreak
        
   Although the injectivity and the dense range of the Newton-type derivative \eqref{definition:J_F_eps}~of~${\mathtt{F}_{\varepsilon}\colon\hspace{-0.1em} \mathbb{A} \hspace{-0.1em}\to \hspace{-0.1em}\mathbb{B}}$, defined~by~\eqref{definition:F_eps}, established in Proposition \ref{prop:J_F_eps}, 
do not imply
invertibility in the infinite-dimen\-sional setting, it still indicates that the loss of invertibility is solely due to the lack of~surjectivity. In the finite-dimensional setting, this obstruction disappears and the Newton derivatives are invertible, even with uniformly bounded inverse. 

In view of these observations, we still propose~the following
semi-smooth Newton method:\vspace{-0.5mm}

\begin{algorithm}[Semi-smooth Newton method for regularized TV-minimization]\label{alg:SSNM-cont}
Let $\varepsilon>0$ be a regularization parameter, let $\gamma>0$ be a proximity parameter, let $\varepsilon_{\mathtt{loc}}^{\mathtt{stop}} >0$ be a stopping~parameter, let $(z^{0},u^{0})\in \mathbb{A}$ be an initial iterate, and let $k_{\mathtt{loc}}^{\mathtt{max}}\in \mathbb{N}\cup\{+\infty\}$ be a number of maximum iterations. Then, for  $k=0,\ldots,\smash{k_{\mathtt{loc}}^{\mathtt{max}}}$, perform the following iteration loop:
\begin{itemize}[noitemsep,topsep=2pt,leftmargin=!,labelwidth=\widthof{(2)}]
\item[(1)] \hypertarget{alg:SSNM-cont.1}{} Compute the primal-dual update direction $(\delta z^k,\delta u^k)\in \mathbb{A}$~such~that\vspace{-0.5mm}
\begin{align*}
  \mathtt{J}_{\mathtt{F}_{\varepsilon}}(z^{k},u^{k})(\delta z^k,\delta u^k) = -\mathtt{F}_{\varepsilon}(z^{k},u^{k})\quad\text{ in }\mathbb{B}\,,\\[-6mm]\notag
\end{align*}
and the updated iterate $(z^{k+1},u^{k+1})\coloneqq (z^{k},u^{k})+(\delta z^k\hspace{-0.1em},\delta u^k)\in  \mathbb{A}$;
\item[(2)] If $\|\mathtt{F}_{\varepsilon}(z^{k+1},u^{k+1})\|_{\mathbb{B}}<\varepsilon_{\mathtt{loc}}^{\mathtt{stop}}$, then \textup{STOP}; otherwise, $k\to k+1$~and~\mbox{continue}~with~(\hyperlink{alg:SSNM-cont.1}{1}). 
\end{itemize}
\end{algorithm}

In the infinite-dimensional setting, Algorithm~\ref{alg:SSNM-cont}
cannot be expected to be globally well-posed, since the Newton-type
derivative \eqref{definition:J_F_eps} is, in general, not known to be
invertible.~In~addition, even if \hspace{-0.1mm}all \hspace{-0.1mm}iterates \hspace{-0.1mm}are \hspace{-0.1mm}computable \hspace{-0.1mm}for \hspace{-0.1mm}a \hspace{-0.1mm}suitable \hspace{-0.1mm}initial
\hspace{-0.1mm}iterate \hspace{-0.1mm}--which \hspace{-0.1mm}is \hspace{-0.1mm}not~\hspace{-0.1mm}\mbox{excluded}~\hspace{-0.1mm}\mbox{\emph{a~\hspace{-0.1mm}priori}--},~\hspace{-0.1mm}in~\hspace{-0.1mm}\mbox{general}, only local super-linear convergence
of Algorithm~\ref{alg:SSNM-cont} can be expected.~In~this~connection, one possible
globalization~approach, besides line-search strategies, is to generate
suitable initial iterates using a semi-implicit discretization of the
$L^2$-gradient flow associated with \eqref{intro:primal_regular}, which has
been shown in \cite{BartelsDieningNochetto2018} to provide a globally well-posed and 
unconditionally strongly~stable~iterative~scheme.

To this end, let $\phi_{\varepsilon}\colon \mathbb{R}\to [0,+\infty)$, for every $s\in \mathbb{R}$, be defined by $\phi_{\varepsilon}(s) \coloneqq (\varepsilon^2+s^2)^{\smash{\frac{1}{2}}}$.\vspace{-0.5mm}

\begin{algorithm}[Semi-implicit discretized $L^2$-gradient flow for regularized TV-minimization]\label{alg:gradflow-cont}
Let $\varepsilon>0$ be a regularization parameter, let $\varepsilon_{\mathtt{glob}}^{\mathtt{stop}} >0$ be a stopping parameter, let $u^{0}\in W^{1,1}(\Omega)\cap L^2(\Omega)$ be an initial iterate, and let $k_{\mathtt{glob}}^{\mathtt{max}}\in \mathbb{N}$ be a number of maximum iterations.~Then,~for $k=0,\ldots,\smash{k_{\mathtt{glob}}^{\mathtt{max}}}$, perform the following iteration loop:
\begin{itemize}[noitemsep,topsep=2pt,leftmargin=!,labelwidth=\widthof{(2)}]
\item[(1)] \hypertarget{alg:gradflow-cont.1}{} Compute the updated iterate $u^{k+1}\in W^{1,1}(\Omega)\cap L^2(\Omega)$ such that for every $v\in W^{1,1}(\Omega)\cap L^2(\Omega)$, there holds
\begin{align*}
  \int_{\Omega}{\mathrm{d}_{\tau}u^{k+1}v\,\mathrm{d}x} +
  \int_{\Omega}{\smash{\frac{\phi_{\varepsilon}'(|\nabla u^{k}|)}{|\nabla u^{k}|}}
  \nabla u^{k+1}\cdot\nabla v\,\mathrm{d}x} +
   \alpha\int_{\Omega}{(u^{k+1}-g)v\,\mathrm{d}x} = 0\,,\\[-6mm]\notag
\end{align*} 
where $\smash{\mathrm{d}_{\tau}u^{k+1}\coloneqq \frac{1}{\tau}(u^{k+1}-u^{k})\in W^{1,1}(\Omega)\cap L^2(\Omega)}$;
\item[(2)] \hypertarget{alg:gradflow-cont.2}{} If \hspace{-0.1mm}$\|\mathtt{F}_{\varepsilon}(z^{k+1},u^{k+1})\|_{\mathbb{B}}\hspace{-0.15em}<\hspace{-0.15em}\varepsilon_{\mathtt{glob}}^{\mathtt{stop}}$, where $z^{k+1} \hspace{-0.15em}\coloneqq \hspace{-0.15em}\phi_{\varepsilon}'(|\nabla u^{k}|)|\nabla u^{k}|^{-1}
  \nabla u^{k+1}\hspace{-0.15em}\in \hspace{-0.15em}W^2_0(\operatorname{div};\Omega)$,~\hspace{-0.1mm}then~\hspace{-0.1mm}\textup{STOP}; otherwise, $k\to k+1$ and continue~with~(\hyperlink{alg:gradflow-cont.1}{1}). 
\end{itemize}
\end{algorithm}

Algorithm \ref{alg:gradflow-cont} is globally well-posed and unconditionally strongly stable.\vspace{-0.5mm}\enlargethispage{2.5mm}

\begin{proposition}\label{prop:strong_stability-cont}
Algorithm \ref{alg:gradflow-cont} is globally well-posed, \textit{i.e.}, if $k^{\mathtt{max}}_{\mathtt{glob}}=+\infty$, 
given an arbitrary initial iterate $u^{0}\in W^{1,1}(\Omega)\cap L^2(\Omega)$, all iterates $\{u^{k}\}_{k\in \mathbb{N}}\subseteq W^{1,1}(\Omega)\cap L^2(\Omega)$ are computable, and  unconditionally strongly stable, \textit{i.e.}, for every $K\in \mathbb{N}$, we have that\vspace{-0.5mm}
\begin{align}\label{prop:strong_stability.0}
  I^{\varepsilon}(u^{K})
  + \tau\sum_{k=1}^{K}\int_{\Omega}{\vert \mathrm{d}_{\tau}u^{k}\vert^{2}\,\mathrm{d}x}
  \le I^{\varepsilon}(u^{0})\,.
\end{align}
\end{proposition}

\begin{proof} See 
\cite[Prop.\ 3.4]{BartelsDieningNochetto2018}.
\end{proof}

\begin{remark}[Termination of Algorithm \ref{alg:gradflow-cont}]
     Algorithm \ref{alg:gradflow-cont} with the employed nonlinear residual stopping criterion may not terminate, \textit{i.e.}, if $k^{\mathtt{max}}_{\mathtt{glob}}=+\infty$, there may not exist~some~${k^*\in \mathbb{N}}$ such that $\smash{\|\mathtt{F}_{\varepsilon}(z^{k^*+1},u^{k^*+1})\|_{\mathbb{B}}}<\varepsilon_{\mathtt{glob}}^{\mathtt{stop}}$, where $z^{k^*+1} \coloneqq \phi_{\varepsilon}'(|\nabla u^{k^*}|)|\nabla u^{k^*}|^{-1}
  \nabla u^{k^*+1}\in W^2_0(\operatorname{div};\Omega)$. However, Algorithm~\ref{alg:gradflow-cont} in conjunction with a primal incremental stopping criterion  terminates, \textit{i.e.}, if $\smash{k^{\mathtt{max}}_{\mathtt{glob}}=+\infty}$,  there exists $k^*\in \mathbb{N}$ such that $\| \mathrm{d}_\tau \smash{u^{k^*+1}}\|_{L^2(\Omega)}<\smash{\varepsilon_{\mathtt{glob}}^{\mathtt{stop}}}$.
\end{remark}\newpage

\subsection{Extension to general convex variational problems}\label{sec:other-problems}\vspace{-0.5mm}

\hspace{5mm}The derivation of the  $\operatorname{prox}$-based semi-smooth Newton method for the regularized TV-minimization problem
in Subsection \ref{subsec:rof_continuous} extends to a broad class of convex variational~problems and has some remarkable features.\enlargethispage{6mm}\vspace*{-1.5mm}

\subsubsection{General convex variational problems}\label{subsec:general-problem}\vspace{-0.5mm}

\hspace{5mm}Let a general convex minimization problem be given via the minimization of the primal energy functional $I\colon W^{1,p}_D(\Omega)\to \mathbb{R}\cup\{+\infty\}$, $p\in [1,+\infty]$, for every $v\in W^{1,p}_D(\Omega)$ defined by\vspace{-0.5mm}
\begin{align}\label{subsec:general-problem-primal}
  I(v) \coloneqq \int_{\Omega}{\phi(\cdot,\nabla v)\,\mathrm{d}x} + \int_{\Omega}{\psi(\cdot,v)\,\mathrm{d}x}\,,\\[-6mm]\notag
\end{align}
where the energy densities $\phi\colon \Omega\times\mathbb{R}^d\to \mathbb{R}\cup \{+\infty\}$ and $\psi\colon \Omega\times \mathbb{R}\to \mathbb{R}\cup \{+\infty\}$~are~convex~normal\linebreak integrands 
having the lower compactness property (see \cite[Sec.\ 3]{BartelsKaltenbach2026}, for more details). 
Moreover,\linebreak $W^{1,p}_D(\Omega)\hspace{-0.15em}\coloneqq\hspace{-0.15em} \{v\hspace{-0.15em}\in\hspace{-0.15em} W^{1,p}(\Omega)\mid v=0\text{ \hspace{-0.125mm}a.e.\ \hspace{-0.125mm}on \hspace{-0.125mm}}\Gamma_D\}$ \hspace{-0.125mm}for \hspace{-0.125mm}a \hspace{-0.125mm}relatively~\hspace{-0.125mm}open~\hspace{-0.125mm}\mbox{Dirichlet}~\hspace{-0.125mm}boundary~\hspace{-0.125mm}part~\hspace{-0.125mm}${\Gamma_D\hspace{-0.15em}\subseteq \hspace{-0.15em}\partial\Omega}$.

Under suitable additional assumptions on these energy densities $\phi$ and $\psi$ (\textit{cf}.\ \cite[Sec.\ 3]{BartelsKaltenbach2026}), 
a (Fenchel) dual problem (in the sense of \cite[Rem.\ 4.2, p.\ 60/61]{EkelandTemam1999}) is given via the maximization~of~the dual energy functional $D\colon \smash{W^{p'}_N(\operatorname{div};\Omega)}\to \mathbb{R}\cup\{-\infty\}$, for every $y\in  \smash{W^{p'}_N(\operatorname{div};\Omega)}$ defined by\vspace{-0.5mm}
\begin{align}\label{subsec:general-problem-dual}
  D(y) \coloneqq -\int_{\Omega}{\phi^{*}(\cdot,y)\,\mathrm{d}x} - \int_{\Omega}{\psi^{*}(\cdot,\operatorname{div} y)\,\mathrm{d}x}\,,\\[-6mm]\notag
\end{align}
where $\smash{\phi^*\colon \Omega\times\mathbb{R}^d\to \mathbb{R}\cup \{+\infty\}}$ and $\smash{\psi^*\colon \Omega\times \mathbb{R}\to \mathbb{R}\cup \{+\infty\}}$ denote the Fenchel conjugates (with respect to the second argument) of the energy densities $\phi$ and $\psi$, respectively. Moreover, $\smash{W^{p'}_N(\operatorname{div};\Omega)}\coloneqq \{y\in (L^{p'}(\Omega))^d\mid \operatorname{div}y\in L^{p'}(\Omega)\,,\;\langle y\cdot n,v\rangle_{\smash{W^{1,1-\smash{\frac{1}{p}},p}(\partial\Omega)}}=0\text{ for all } v\in \smash{W^{1,p}_D(\Omega)}\}$ for a relatively open Neumann boundary part $\Gamma_N\subseteq \partial\Omega$ such that $\overline{\Gamma}_D\cup \overline{\Gamma}_N=\partial\Omega$.

A minimizer $u\in \smash{W^{1,p}_D(\Omega)}$ of the primal energy functional \eqref{subsec:general-problem-primal}, called \emph{primal solution},~and a maximizer $z\in \smash{W^{p'}_N(\operatorname{div};\Omega)}$ of the dual energy functional \eqref{subsec:general-problem-dual}, called \emph{dual solution},~are~character\-ized and related by the \emph{primal optimality inclusions}\vspace{-0.5mm}
\begin{subequations}\label{eq:general-inclusions}
\begin{alignat}{2}
  z &\in \partial_{t}\phi(\cdot,\nabla u)&&\quad\text{ a.e.\ in }\Omega\,,\\
  \operatorname{div} z &\in \partial_{s}\psi(\cdot,u)&&\quad\text{ a.e.\ in }\Omega\,,
\end{alignat}\\[-4.5mm]\notag
\end{subequations}
where $\smash{\partial_{t}\phi\colon \Omega\times\mathbb{R}^d\to 2^{\mathbb{R}^d}}$ and $\smash{\partial_{s}\psi\colon \Omega\times\mathbb{R}\to 2^{\mathbb{R}}}$ denote the subdifferentials (with respect to the second argument) of the energy densities $\phi$ and $\psi$, respectively. Note that 
the primal optimality inclusions \eqref{eq:general-inclusions}
are equivalent to the strong duality relation $I(u) = D(z)$ (\textit{cf}.\ \cite[p.\ 2255]{BartelsKaltenbach2023}).

Multiplying the optimality inclusions \eqref{eq:general-inclusions} by proximity parameters $\gamma_{1},\gamma_{2}>0$ and adding $\nabla u$ and $u$ to both sides, respectively, we obtain the equivalent inclusions\vspace{-0.5mm}
\begin{subequations}\label{eq:equivalent-inclusions}
\begin{alignat}{2}
  \nabla u + \gamma_{1}z &\in (\operatorname{id}_{\mathbb{R}^d}+ \gamma_{1}\partial_{t}\phi)(\cdot,\nabla u)&&\quad\text{ a.e.\ in }\Omega\,,\\
 u + \gamma_{2}\operatorname{div} z &\in (\operatorname{id}_{\mathbb{R}} + \gamma_{2}\partial_{s}\psi)(\cdot,u)&&\quad\text{ a.e.\ in }\Omega\,.
\end{alignat}\\[-4.5mm]\notag
\end{subequations} 
Incorporating the proximity operators (with respect to the second argument) $\operatorname{prox}_{\gamma_{1}\phi}\coloneqq (\mathrm{id}_{\mathbb{R}^d}
+\gamma_1\partial_t\phi)^{-1}\colon \hspace{-0.15em}\Omega\times \mathbb{R}^d\hspace{-0.15em}\to\hspace{-0.15em} \mathbb{R}^d$ and $\operatorname{prox}_{\gamma_{2}\psi}\coloneqq(\mathrm{id}_{\mathbb{R}}
+\gamma_2\partial_s\psi)^{-1} \colon\hspace{-0.15em}\Omega\times \mathbb{R}\hspace{-0.15em}\to\hspace{-0.15em} \mathbb{R}$ of the energy~\mbox{densities}~$\phi$~and~$\psi$, respectively, 
the inclusions \eqref{eq:equivalent-inclusions} may be re-written  equivalently~as~the~equations\vspace{-0.5mm}
\begin{subequations}\label{eq:equivalent-inclusions.2}
\begin{alignat}{2}\label{eq:equivalent-inclusions.2.1}
    \nabla u &= \operatorname{prox}_{\gamma_{1}\phi}(\cdot,\nabla u + \gamma_{1} z)&&\quad \text{ a.e.\ in }\Omega\,,\\
    u &= \operatorname{prox}_{\gamma_{2}\psi}(\cdot,u + \gamma_{2}\operatorname{div} z)&&\quad \text{ a.e.\ in }\Omega\,.\label{eq:equivalent-inclusions.2.2}
\end{alignat}\\[-4.5mm]\notag
\end{subequations}
As a consequence, the primal optimality inclusions \eqref{eq:general-inclusions} are equivalent to the root finding problem for the nonlinear mapping
$\mathtt{F}\colon \smash{W^{p'}_N(\operatorname{div};\Omega)}\times W^{1,p}_D(\Omega)\to (L^{\min\{p,p'\}}(\Omega))^d\times L^{\min\{p,p'\}}(\Omega)$, for every $(y,v)\in \smash{W^{p'}_N(\operatorname{div};\Omega)}\times W^{1,p}_D(\Omega)$ defined by\vspace{-0.5mm}
\begin{align}\label{def:F-general}
    \mathtt{F}(y,v) \coloneqq
  \begin{bmatrix}
    \nabla v - \operatorname{prox}_{\gamma_{1}\phi}(\cdot,\nabla v + \gamma_{1} y)\\
    v - \operatorname{prox}_{\gamma_{2}\psi}(\cdot,v + \gamma_{2}\operatorname{div} y)
  \end{bmatrix}\quad\text{ a.e.\ in }\Omega\,,\\[-6mm]\notag
\end{align}
which, by the equations \eqref{eq:equivalent-inclusions.2}, admits the root $(z,u)\in \smash{W^{p'}_N(\operatorname{div};\Omega)\times W^{1,p}_D(\Omega)}$, \textit{i.e.},~we~have~that\vspace{-0.5mm}
\begin{align*}
     \mathtt{F}(z,u) = (\mathtt{0}_d,0) \quad \text{ a.e.\ in }\Omega\,.\\[-6mm]\notag
\end{align*}
Moreover, for a.e.\ $x\hspace{-0.1em}\in\hspace{-0.1em} \Omega$, the proximity operators $\operatorname{prox}_{\gamma_{1}\phi}(x,\cdot)\colon \hspace{-0.1em} \mathbb{R}^d\hspace{-0.1em}\to\hspace{-0.1em} \mathbb{R}^d$ and $\operatorname{prox}_{\gamma_{2}\psi}(x,\cdot) \colon \hspace{-0.1em}\mathbb{R}\hspace{-0.1em}\to\hspace{-0.1em} \mathbb{R}$  are Lipschitz continuous with constant $1$ and, consequently, due to \cite[Thm.~2.6]{ChenNashedQi2000}, point-wise Newton differentiable with point-wise Newton derivatives $\mathtt{J}_{\smash{\operatorname{prox}_{\gamma_1\phi}}}(x,\cdot)\colon \smash{B_{\delta_{x,t}}^d(t)}\to \mathbb{R}^{d\times d}$, $\delta_{x,t}>0$, for all $t\in \mathbb{R}^d$ and $\mathtt{J}_{\smash{\operatorname{prox}_{\gamma_2\psi}}}(x,\cdot)\colon \smash{B_{\delta_{x,s}}^1(s)}\to \mathbb{R}$, $\delta_{x,s}>0$, for all $s\in \mathbb{R}$, respectively.~In~what~follows,\linebreak for a.e.\ $x\in \Omega$, we assume that the above proximity operators 
are even globally Newton differentiable with (global)
Newton derivatives $\mathtt{J}_{\smash{\operatorname{prox}_{\gamma_1\phi}}}(x,\cdot)\colon \hspace{-0.1em}\mathbb{R}^d\hspace{-0.1em}\to\hspace{-0.1em} \mathbb{R}^{d\times d}$ and ${\mathtt{J}_{\smash{\operatorname{prox}_{\gamma_2\psi}}}(x,\cdot)\colon \hspace{-0.1em}\mathbb{R}\hspace{-0.1em}\to\hspace{-0.1em} \mathbb{R}}$,~which is not generally true, but typically satisfied in practice and can be verified on~a~\emph{\mbox{case-by-case}~basis}.

By analogy with Subsection \ref{subsec:rof_continuous}, the  mapping
$\mathtt{F}\colon W^{p'}_N(\operatorname{div};\Omega)\times W^{1,p}_D(\Omega)\to (L^{\min\{p,p'\}}(\Omega))^d\times L^{\min\{p,p'\}}(\Omega)$, defined by \eqref{def:F-general}, is point-wise Newton differentiable with a point-wise Newton derivative at a root that may depend on this root. Similarly, a natural candidate for a Newton derivative is given via the Newton-type derivative  $\mathtt{J}_{\mathtt{F}}\colon W^{p'}_N(\operatorname{div};\Omega)\times W^{1,p}_D(\Omega)\to \mathcal{L}(W^{p'}_N(\operatorname{div};\Omega)\hspace{-0.15em}\times \hspace{-0.15em}W^{1,p}_D\hspace{-0.175mm}(\Omega);(L^{\min\{p,p'\}}\hspace{-0.175mm}(\Omega))^d\times L^{\min\{p,p'\}}\hspace{-0.175mm}(\Omega))$, \hspace{-0.175mm}for \hspace{-0.175mm}every \hspace{-0.175mm}${(y,v),\hspace{-0.1em}(\widehat{y},\widehat{v})\hspace{-0.2em}\in \hspace{-0.2em}W^{p'}_N\hspace{-0.175mm}(\operatorname{div};\Omega)\hspace{-0.225em}\times \hspace{-0.225em}W^{1,p}_D\hspace{-0.175mm}(\Omega)}$~\hspace{-0.175mm}\mbox{defined}~\hspace{-0.175mm}by\vspace{-0.5mm}
\begin{align}\label{def:J_F-general}
  \mathtt{J}_{\mathtt{F}}(y,v)(\widehat{y},\widehat{v}) =
  \begin{bmatrix}
    (\mathbbone - \mathtt{J}_{\operatorname{prox}_{\gamma_{1}\phi}}(a))\nabla\widehat{v}-\gamma_{1} \mathtt{J}_{\operatorname{prox}_{\gamma_{1}\phi}}(a)\widehat{y}  \\
    (\operatorname{1} - \mathtt{J}_{\operatorname{prox}_{\gamma_{2}\psi}}(b))\widehat{v}-\gamma_{2} \mathtt{J}_{\operatorname{prox}_{\gamma_{2}\psi}}(b)\operatorname{div} \widehat{y}
  \end{bmatrix}\quad\text{ a.e.\ in }\Omega\,,\\[-6mm]\notag
\end{align}
where $a\coloneqq \nabla v + \gamma_{1}y\in (L^{\min\{p,p'\}}(\Omega))^d$ and $ b\coloneqq v + \gamma_{2}\operatorname{div} y\in L^{\min\{p,p'\}}(\Omega)$.

Similarly \hspace{-0.1mm}to \hspace{-0.1mm}Algorithm~\hspace{-0.1mm}\ref{alg:SSNM-cont}, \hspace{-0.1mm}although \hspace{-0.1mm}the \hspace{-0.1mm}Newton-type
\hspace{-0.1mm}derivative~\hspace{-0.1mm}\eqref{def:J_F-general}, \hspace{-0.1mm}in \hspace{-0.1mm}general,~\hspace{-0.1mm}may~\hspace{-0.1mm}not be \hspace{-0.1mm}a \hspace{-0.1mm}genuine \hspace{-0.1mm}Newton
\hspace{-0.1mm}derivative, \hspace{-0.1mm}at \hspace{-0.1mm}least \hspace{-0.1mm}not \hspace{-0.1mm}with \hspace{-0.1mm}respect \hspace{-0.1mm}to \hspace{-0.1mm}the \hspace{-0.1mm}function \hspace{-0.1mm}spaces \hspace{-0.1mm}under \hspace{-0.1mm}consideration, introducing the abbreviated notation for the product spaces\vspace{-0.5mm}
\[
\smash{\mathbb{A}_p\coloneqq 
W^p_N(\operatorname{div};\Omega)
\times
W^{1,p}_D(\Omega)\,,
\qquad
\mathbb{B}_p
\coloneqq
(L^p(\Omega))^d\times L^p(\Omega)\,,}\\[-0.5mm]\notag
\]
we
nevertheless formulate a semi-smooth Newton method based on this~\mbox{Newton-type}~\mbox{derivative}:\enlargethispage{3mm}

\begin{algorithm}[Semi-smooth Newton method for general convex minimization problems]\label{alg:SSNM-cont-general}
Let $\gamma_1,\gamma_2>0$ be proximity parameters, let $\varepsilon_{\mathtt{loc}}^{\mathtt{stop}} >0$ be a stopping parameter, let $(z^{0},u^{0})\in \mathbb{A}_p$ be an initial iterate, and let $k_{\mathtt{loc}}^{\mathtt{max}}\in \mathbb{N}\cup\{+\infty\}$ be a number of maximum iterations. Then, for $k=0,\ldots,\smash{k_{\mathtt{loc}}^{\mathtt{max}}}$, perform the following iteration loop:
\begin{itemize}[noitemsep,topsep=2pt,leftmargin=!,labelwidth=\widthof{(2)}]
\item[(1)] \hypertarget{alg:SSNM-cont.1}{} Compute the primal-dual update direction $(\delta z^k,\delta u^k) \in \mathbb{A}_p$ such that\vspace{-0.5mm}
\begin{align}\label{alg:SSNM-cont-general.0}
  \mathtt{J}_{\mathtt{F}}(z^{k},u^{k})(\delta z^k,\delta u^k) = -\mathtt{F}(z^{k},u^{k})\quad\text{ in }\mathbb{B}_p\,,\\[-6mm]\notag
\end{align}
and the updated iterate $\smash{(z^{k+1},u^{k+1}) \coloneqq (z^{k},u^{k}) +(\delta z^k,\delta u^k)\in  \mathbb{A}_p}$;
\item[(2)] If $\|\mathtt{F}(z^{k+1},u^{k+1})\|_{\mathbb{B}_p}<\varepsilon_{\mathtt{loc}}^{\mathtt{stop}}$, then \textup{STOP}; otherwise, $k\to  k+1$ and~\mbox{continue}~with~(\hyperlink{alg:SSNM-cont.1}{1}). 
\end{itemize}
\end{algorithm}

The perturbed saddle-point problems arising in each iteration
step, under suitable assumptions, can be reformulated as linear elliptic
problems, which we  intend to solve in the space $W^{1,p}_D(\Omega)$.\linebreak In order to
ensure well-posedness of these problems, we do not seek the
dual iterates~in~$\smash{W^{p'}_N(\operatorname{div};\Omega)}$, but rather in
$W^p_N(\operatorname{div};\Omega)$, \hspace{-0.1mm}so \hspace{-0.1mm}that \hspace{-0.1mm}the \hspace{-0.1mm}right-hand \hspace{-0.1mm}sides \hspace{-0.1mm}in \hspace{-0.1mm}the \hspace{-0.1mm}linear \hspace{-0.1mm}elliptic
\hspace{-0.1mm}problems~\hspace{-0.1mm}lie~\hspace{-0.1mm}in~\hspace{-0.1mm}$W^{-1,p}_D(\Omega)$.\linebreak By the classical Gröger regularity theory (\textit{cf}.\ \cite{Groeger1989}), however, the well-posedness can only be expected for $p$ in a small interval around~$2$. Under additional regularity assumptions, by means of the classical Schauder regularity theory (\textit{cf}.\ \cite[Sec.\ 8.11]{GilbargTrudinger2001}),~this~interval~can~be~extended~to~$(1,\infty)$.

\begin{theorem}[Well-posedness of Algorithm \ref{alg:SSNM-cont-general}]\label{thm:SSNM-cont-general}
If there exist constants $0< \underline{\mu}\leq \overline{\mu}<1$~such~that for every $a\in (L^p(\Omega))^d$ and $b\in L^p(\Omega)$, there holds\vspace{-0.5mm}
\begin{subequations} \label{thm:SSNM-cont-general.0} 
\begin{alignat}{2}\label{thm:SSNM-cont-general.0.1} 
    \underline{\mu}\mathbbone&\preceq \smash{\mathtt{J}_{\operatorname{prox}_{\gamma_{1}\phi}}(a)}\preceq \overline{\mu}\mathbbone&&\quad\text{ a.e.\ in }\Omega\,,\\
    \underline{\mu}\mathbbone&\leq \smash{\mathtt{J}_{\operatorname{prox}_{\gamma_{2}\psi}}(b)}&&\quad\text{ a.e.\ in }\Omega\,, \label{thm:SSNM-cont-general.0.2} \\[-6mm]\notag
\end{alignat}
\end{subequations}
then the following statements apply:
\begin{itemize}[noitemsep,topsep=2pt,leftmargin=!,labelwidth=\widthof{(ii)}]
    \item[(i)]\hypertarget{thm:SSNM-cont-general.i}{} If  $(\Omega,\Gamma_D)$ is Gröger-regular (\textit{cf}.\ \cite{Groeger1989}), there exists $\varepsilon>0$ such that if ${p\in (2-\varepsilon,2+\varepsilon)}$,~Algori\-thm \ref{alg:SSNM-cont-general} is globally well-posed, \textit{i.e.}, if $k_{\mathtt{loc}}^{\mathtt{max}}\hspace{-0.15em}=\hspace{-0.15em}+\infty$, given an arbitrary initial~iterate~${(z^0,u^0)\hspace{-0.15em}\in\hspace{-0.15em} \mathbb{A}_p}$, 
    all iterates  $\{(z^k,u^k)\}_{k\in \mathbb{N}}\subseteq  \mathbb{A}_p$~are~\mbox{computable}.

    \item[(ii)]\hypertarget{thm:SSNM-cont-general.ii}{} If $(\Omega,\Gamma_D)$ is $C^{1,\alpha}$-regular, $\alpha\hspace{-0.1em}\in \hspace{-0.1em}(0,1)$,  and $\smash{\mathtt{J}_{\operatorname{prox}_{\gamma_{1}\phi}}(a)}\hspace{-0.1em}\in\hspace{-0.1em} (C^{0,\alpha}(\overline{\Omega}_D))^{d\times d}$~for~all~${a\hspace{-0.1em}\in\hspace{-0.1em} (C^{0,\alpha}(\overline{\Omega}_D))^d}$, $\Omega_D\coloneqq \Omega\cup \Gamma_D$,  Algorithm \ref{alg:SSNM-cont-general} is well-posed on $\mathbb{A}_p\cap \mathbb{C}_\alpha$, where ${\mathbb{C}_\alpha\coloneqq(C^{0,\alpha}(\overline{\Omega}_D))^d\times C^{1,\alpha}(\overline{\Omega}_D)}$,
    \textit{i.e.}, if $k_{\mathtt{loc}}^{\mathtt{max}}=+\infty$ and  $(z^0,u^0)\in \mathbb{A}_p\cap \mathbb{C}_\alpha$, all iterates $\{(z^k,u^k)\}_{k\in \mathbb{N}}\subseteq  \mathbb{A}_p\cap \mathbb{C}_\alpha$~are~\mbox{computable}.
\end{itemize}

\end{theorem}

\begin{remark}
    Sufficient conditions for the lower and upper bounds in \eqref{thm:SSNM-cont-general.0} are the following:
    \begin{itemize}[noitemsep,topsep=2pt,leftmargin=!,labelwidth=\widthof{$\bullet$}]
        \item[$\bullet$] \emph{Lower bound.} If $\partial_t\phi\colon \Omega\times \mathbb{R}^d\to 2^{\smash{\mathbb{R}^d}}$ and  $\partial_s\psi\colon \Omega\times \mathbb{R}\to 2^{\mathbb{R}}$ are (uniformly with respect~to~$x\in \Omega$)\linebreak $\smash{\frac{\underline{\mu}}{1-\underline{\mu}}}$-cocoercive  (\textit{cf}.\ \cite[Def.\ 4.10]{BauschkeCombettes2017}), then, for every $a\in (L^p(\Omega))^d$ and $b\in L^p(\Omega)$, there holds $\smash{\mathtt{J}_{\operatorname{prox}_{\gamma_{1}\phi}}(a)}\succeq \underline{\mu}$ and $\smash{\mathtt{J}_{\operatorname{prox}_{\gamma_{2}\psi}}(b)}\succeq \underline{\mu}$ a.e.\ in $\Omega$;
        \item[$\bullet$] \emph{Upper bound.} If $\phi\colon \Omega\times \mathbb{R}^d\to \mathbb{R}\cup\{+\infty\}$ is (uniformly with respect~to~$x\in \Omega$) $\smash{\overline{\mu}}$-strongly convex  (\textit{cf}.\ \cite[Def.\ 10.7]{BauschkeCombettes2017}), then, for every $a\in \smash{(L^p(\Omega))^d}$, there holds $\smash{\mathtt{J}_{\operatorname{prox}_{\gamma_{1}\phi}}(a)}\preceq\overline{\mu}$~a.e.~in~$\Omega$.
    \end{itemize}
\end{remark}

\begin{proof}[Proof (of Theorem \ref{thm:SSNM-cont-general}).] \textit{ad  (\hyperlink{thm:SSNM-cont-general.i}{i}).}
    It suffices to establish that for every $(y,v)\in \mathbb{A}_p$,~the~\mbox{Newton-type} derivative $\mathtt{J}_{\mathtt{F}}(y,v)\in  \mathcal{L}(\mathbb{A}_p;\mathbb{B}_p)$, defined by \eqref{def:J_F-general}, is invertible. To this end, let 
    $(y,v)\in  \mathbb{A}_p$ and $(f_1,f_2)\in  \mathbb{B}_p$~be~fixed, but arbitrary. Then, we seek $(\delta y,\delta v)\in \mathbb{A}_p$ such that 
    \begin{align}\label{thm:SSNM-cont-general.1}
        \mathtt{J}_{\mathtt{F}}(y,v)(\delta y,\delta v)=(f_1,f_2)\quad \text{ in }\mathbb{B}_p\,. 
    \end{align}
    First, we note that the system  \eqref{thm:SSNM-cont-general.1}
    reads 
    \begin{subequations} \label{thm:SSNM-cont-general.2}
    \begin{alignat}{2}\label{thm:SSNM-cont-general.2.1}
         \delta y &= \smash{\tfrac{1}{\gamma_1}\mathtt{J}_1^{-1}\{(1-\mathtt{J}_1)\nabla \delta v+f_1\}}&&\quad \text{ a.e.\ in }\Omega\,,\\
         \operatorname{div} \delta y&=\smash{\tfrac{1}{\gamma_2}\mathtt{J}_2^{-1}\{(1-\mathtt{J}_2)\delta v+f_2\}}&&\quad \text{ a.e.\ in }\Omega\,,\label{thm:SSNM-cont-general.2.2}
    \end{alignat}
    \end{subequations}
    where (again, abbreviating $a\coloneqq\nabla v+\gamma_1 y\in (L^p(\Omega))^d$ and $b\coloneqq v+\gamma_2 \operatorname{div}y\in  L^p(\Omega)$ as in \eqref{def:J_F-general}) the invertibility of $\mathtt{J}_1\coloneqq  \smash{\mathtt{J}_{\smash{\operatorname{prox}_{\gamma_{1}\phi}}}(a)}\in  (L^\infty(\Omega))^{d\times d}$ and ${\mathtt{J}_2 \coloneqq  \smash{\mathtt{J}_{\smash{\operatorname{prox}_{\gamma_{2}\psi}}}(b)}\in  L^\infty(\Omega)}$ is~based~on~\eqref{thm:SSNM-cont-general.0}.  
    Next, inserting \eqref{thm:SSNM-cont-general.2.1} in \eqref{thm:SSNM-cont-general.2.2} and rearranging, we arrive at the linear~elliptic~problem 
    \begin{align}\label{thm:SSNM-cont-general.3}
        \hspace{-1.5mm}\smash{- \operatorname{div}\bigl(\tfrac{1}{\gamma_1}\mathtt{J}_1^{-1}(1\hspace{-0.12em}-\hspace{-0.12em}\mathtt{J}_1)\nabla \delta v\bigr)\hspace{-0.12em}+\hspace{-0.12em}\tfrac{1}{\gamma_2}\mathtt{J}_2^{-1}(1\hspace{-0.12em}-\hspace{-0.12em}\mathtt{J}_2)\delta v\hspace{-0.12em}=\hspace{-0.12em}-\operatorname{div}\bigl(\tfrac{1}{\gamma_1}\mathtt{J}_1^{-1}f_1\bigr)\hspace{-0.12em}+\hspace{-0.12em}\tfrac{1}{\gamma_2}\mathtt{J}_2^{-1}f_2\quad \text{ in }W^{-1,p}(\Omega)\,.}
    \end{align}
    Since, due to \eqref{thm:SSNM-cont-general.0}, we have that $\mathtt{J}_1^{-1}(1-\mathtt{J}_1)\in (L^\infty(\Omega))^{d\times d}$ and $\mathtt{J}_2^{-1}(1-\mathtt{J}_2)\in L^\infty(\Omega)$ with
    \begin{subequations}\label{thm:SSNM-cont-general.4}
    \begin{alignat}{2}
       ( \smash{\tfrac{1}{\overline{\mu}}}-1)\mathbbone&\preceq\mathtt{J}_1^{-1}(1-\mathtt{J}_1)\preceq ( \smash{\tfrac{1}{\underline{\mu}}}-1)\mathbbone&&\quad\text{ a.e.\ in }\Omega\,,\\ 
        0&\le\mathtt{J}_2^{-1}(1-\mathtt{J}_2)&&\quad\text{ a.e.\ in }\Omega\,, 
    \end{alignat}
    \end{subequations}
    according to the classical Gröger regularity theory (\textit{cf}.\ \cite[Thm.\ 1]{Groeger1989}), there exists $\delta v\in W^{1,p}_D(\Omega)$, solving \eqref{thm:SSNM-cont-general.3}, and, consequently, $\delta y\in W^p_N(\operatorname{div};\Omega)$, defined by \eqref{thm:SSNM-cont-general.2.1}. 

    \textit{ad (\hyperlink{thm:SSNM-cont-general.ii}{ii}).} It suffices to establish that for every $(y,v)\in \mathbb{A}_p\cap \mathbb{C}_\alpha$,~the restricted~(not~\mbox{re-labelled}) Newton-type derivative  $\mathtt{J}_{\mathtt{F}}(y,v)\hspace{-0.15em}\in\hspace{-0.15em} \mathcal{L}(\mathbb{A}_p \cap \mathbb{C}_\alpha;(C^{0,\alpha}(\overline{\Omega}_D))^d\hspace{-0.05em}\times\hspace{-0.05em} L^\infty(\Omega))$, defined by \eqref{def:J_F-general},~is~\mbox{invertible}. To~this~end,~let 
    $(y,v)\in \mathbb{A}_p\cap \mathbb{C}_\alpha$ and $(f_1,f_2)\in (C^{0,\alpha}(\overline{\Omega}_D))^d\times L^\infty(\Omega)$ be fixed, but arbitrary. Then, we seek $(\delta y,\delta v)\in \mathbb{A}_p\cap \mathbb{C}_\alpha $ solving \eqref{thm:SSNM-cont-general.2},  which, again, is equivalent~to~\eqref{thm:SSNM-cont-general.3}.~Since, by \eqref{thm:SSNM-cont-general.0} and the additional
    assumption  $\mathtt{J}_1\in  (C^{0,\alpha}(\overline{\Omega}_D))^{d\times d}$, we have that  $
    \mathtt{J}_1^{-1}(1-\mathtt{J}_1)\in (C^{0,\alpha}(\overline{\Omega}_D))^{d\times d}$ and $\mathtt{J}_2^{-1}(1-\mathtt{J}_2)\in  L^\infty(\Omega)$ with \eqref{thm:SSNM-cont-general.4} as well as $\mathtt{J}_1^{-1}f_1\hspace{-0.1em}\in\hspace{-0.1em} (C^{0,\alpha}(\overline{\Omega}_D))^d$ and $\mathtt{J}_2^{-1}f_2\hspace{-0.1em}\in\hspace{-0.1em} L^\infty(\Omega)$,~according~to the classical Schauder~\mbox{regularity} theory (\textit{cf}.\ \cite[Cor.\ 8.36]{GilbargTrudinger2001}),~there~\mbox{exists}~${\delta v\in  W^{1,p}_D(\Omega)\cap C^{1,\alpha}(\overline{\Omega}_D)}$, solving~\eqref{thm:SSNM-cont-general.3}, and, consequently, $\delta y\in W^p_N(\operatorname{div};\Omega)\cap(C^{0,\alpha}(\overline{\Omega}_D))^d$, defined by \eqref{thm:SSNM-cont-general.2.1}.
\end{proof}

\subsubsection{Obstacle problems}\enlargethispage{0.5mm}

\hspace{5mm}Semi-smooth Newton methods are well-established for obstacle problems (\textit{cf}.\ \cite{HintermüllerItoKunisch2003,HintermüllerKunisch2004}), where the energy density $\psi\colon \Omega\times \mathbb{R}\to \mathbb{R}\cup\{+\infty\}$ involves an indicator functional of a convex set that encodes the obstacle constraint. In the simplest case, the energy density $\psi\colon \Omega\times \mathbb{R}\to \mathbb{R}\cup\{+\infty\}$, for a.e.\ $x\in \Omega$ and every $s\in \mathbb{R}$, is of the form
\begin{align}\label{eq:simple-obstacle}
  \psi(x,s) \coloneqq I_{[0,+\infty)}(s) - f(x)s \coloneqq
  \begin{cases}
    -f(x)s &\text{ if } s\ge 0\,,\\
    +\infty &\text{ else}\,,
  \end{cases}
\end{align}
where $f\in L^2(\Omega)$ usually represents a vertical force acting on the membrane fixed at a Dirichlet boundary portion $\Gamma_D\hspace{-0.15em}\subseteq\hspace{-0.15em} \partial\Omega$, the deflection of which is represented by $u\hspace{-0.15em}\in\hspace{-0.15em} W^{1,2}_D(\Omega)$.~In~the~case~\eqref{eq:simple-obstacle}, 
the associated proximity operator $\operatorname{prox}_{\gamma_{2}\psi}\colon \Omega\times \mathbb{R}\to \mathbb{R}$, for a.e.\ $x\in \Omega$ and every $s\in \mathbb{R}$,~is~given~via $\operatorname{prox}_{\gamma_{2}\psi}(x,s) = \max\{ 0, s+f(x)\}$,
so that the identity 
\eqref{eq:equivalent-inclusions.2.2}~becomes 
\begin{align*}
  u - \max\{0, u + \gamma_{2}(\operatorname{div}z + f)\} = 0\quad\text{ a.e.\ in }\Omega\,. 
\end{align*}
Eliminating $z = \nabla u$, if the energy density $\phi\colon \Omega\times\smash{\mathbb{R}^d}\to \mathbb{R}\cup\{+\infty\}$, for a.e.\ $x\in \Omega$ and every $t\in \mathbb{R}^d$, is defined by $\phi(x,t) \coloneqq \frac{1}{2}|t|^{2}$, 
and introducing the Lagrange multiplier $\lambda \coloneqq \operatorname{div} z + f\in L^2(\Omega)$~if $\Delta u\in L^2(\Omega)$, this corresponds to the typical formulation of the complementarity conditions for the obstacle problem within semi-smooth Newton methods (\textit{cf}.~\cite{HintermüllerItoKunisch2003,HintermüllerKunisch2004}).

\subsubsection{Duality invariance}\label{subsec:dual-invariance}

\hspace{5mm}In the case of a strong duality relation (\textit{i.e.}, $I(u)=D(z)$, 
 which is
equivalent to the primal optimality inclusions \eqref{eq:general-inclusions}), 
it is meaningless to particularly distinguish the roles~of~the~primal and the dual problem,~and this aspect is reflected in the mapping $\mathtt{F}\colon \smash{W^{p'}_N(\operatorname{div};\Omega)}\times W^{1,p}_D(\Omega)\to (L^{\min\{p,p'\}}(\Omega))^d\times L^{\min\{p,p'\}}(\Omega)$, defined by \eqref{def:F-general}: The dual optimality inclusions are given via\vspace{-0.5mm}
\begin{subequations}\label{subsec:dual-invariance-inclusions}
\begin{alignat}{2}
  \nabla u &\in \partial_{t}\phi^{*}(\cdot,z)&&\quad \text{ a.e.\ in }\Omega\,,\\
  u &\in \partial_{s}\psi^{*}(\cdot,\operatorname{div} z)&&\quad \text{ a.e.\ in }\Omega\,.
\end{alignat}\\[-4.5mm]\notag
\end{subequations}
Multiplying the dual optimality inclusions \eqref{subsec:dual-invariance-inclusions} by proximity parameters $\sigma_{1},\sigma_{2}>0$ and adding $z$ and $\operatorname{div} z$, respectively, we obtain the equivalent inclusions\vspace{-0.5mm}
\begin{subequations}\label{subsec:dual-invariance-inclusions.2}
\begin{alignat}{2} 
  z + \sigma_{1}\nabla u &\in (\operatorname{id}_{\mathbb{R}^d} + \sigma_{1}\partial_{t}\phi^{*})(\cdot,z)&&\quad \text{ a.e.\ in }\Omega\,,\\
  \operatorname{div} z + \sigma_{2}u &\in (\operatorname{id}_{\mathbb{R}} + \sigma_{2}\partial_{s}\psi^{*})(\cdot,\operatorname{div} z)&&\quad \text{ a.e.\ in }\Omega\,.
\end{alignat}\\[-4.5mm]\notag
\end{subequations}
By incorporating the proximity operators (with respect to the second argument) $\operatorname{prox}_{\sigma_{1}\phi^*}\colon \Omega\times \mathbb{R}^d\to \mathbb{R}^d$ and $\operatorname{prox}_{\smash{\sigma_{2}\psi^*}}\colon \Omega\times \mathbb{R}\to \mathbb{R}$,
we arrive at the equations\vspace{-1mm}
\begin{subequations}\label{subsec:dual-invariance-inclusions.3}
    \begin{alignat}{2}
         z &=  \operatorname{prox}_{\sigma_{1}\phi^*}(z + \sigma_{1}\nabla u)&&\quad \text{ a.e.\ in }\Omega\,,\\[-0.5mm]
         \operatorname{div} z &= \operatorname{prox}_{\sigma_{2}\psi^*}(\operatorname{div} z + \sigma_{2}u)&&\quad \text{ a.e.\ in }\Omega\,.
    \end{alignat}\\[-4.5mm]\notag
\end{subequations}
Hence, the dual optimality inclusions \eqref{subsec:dual-invariance-inclusions} are equivalent to the root finding problem for the nonlinear mapping 
$\mathtt{F}^*\colon \smash{W^{p'}_N(\operatorname{div};\Omega)}\times W^{1,p}_D(\Omega)\to (L^{\min\{p,p'\}}(\Omega))^d\times L^{\min\{p,p'\}}(\Omega)$, for every $(y,v)\in \smash{W^{p'}_N(\operatorname{div};\Omega)}\times W^{1,p}_D(\Omega)$ defined by\vspace{-0.5mm}
\begin{align}\label{subsec:dual-invariance-inclusions.4}
 \mathtt{F}^{*}(y,v) \coloneqq
  \begin{bmatrix}
    y - \operatorname{prox}_{\sigma_{1}\phi^*}(y + \sigma_{1}\nabla v)\\
    \operatorname{div} y - \operatorname{prox}_{\sigma_{2}\psi^*}(\operatorname{div} y + \sigma_{2}v)
  \end{bmatrix}\quad\text{ a.e.\ in }\Omega\,.\\[-6mm]\notag
\end{align}
By Moreau's identity (\textit{cf}.\  \cite[Thm.\ 14.3(ii)]{BauschkeCombettes2017}), for a.e.\ $x\in \Omega$, every $t\in \mathbb{R}^d$, and $s\in \mathbb{R}$,~we~have~that\vspace{-4.25mm}
\begin{subequations}\label{eq:moreau_decomposition}
    \begin{align}
  t &= \gamma_{1} \operatorname{prox}_{\smash{\frac{1}{\gamma_{1}}\phi^*}}(x,\tfrac{1}{\gamma_{1}} t) + \operatorname{prox}_{\gamma_{1}\phi}(x,t)\,,
  \\
  s &= \gamma_{2} \operatorname{prox}_{\smash{\frac{1}{\gamma_{2}}\psi^*}}(x,\tfrac{1}{\gamma_{2}}s) + \operatorname{prox}_{\gamma_{2}\psi}(x,s)\,.
\end{align}
\end{subequations} 
Letting $\sigma_{i} \coloneqq \smash{\frac{1}{\gamma_{i}}}$, $i=1,2$,  $t \coloneqq  \gamma_1\{z + \sigma_{1}\nabla u\}$, and $s \coloneqq\gamma_{2}\{ \operatorname{div} z + \sigma_{2}u\}$ in \eqref{eq:moreau_decomposition},  we find that\vspace{-0.5mm}
\begin{subequations}\label{subsec:dual-invariance-inclusions.5}
    \begin{alignat}{2}
  z + \sigma_{1}\nabla u
  &= \operatorname{prox}_{\sigma_{1}\phi^*}(\cdot,z + \sigma_{1}\nabla u)
    + \smash{\tfrac{1}{\gamma_{1}}} \operatorname{prox}_{\gamma_{1}\phi}(\cdot,\gamma_{1} z + \nabla u)&&\quad\text{ a.e.\ in }\Omega\,,\\
  \operatorname{div} z + \sigma_{2}u
  &= \operatorname{prox}_{\sigma_{2}\psi^*}(\cdot,\operatorname{div} z + \sigma_{2}u)
    + \smash{\tfrac{1}{\gamma_{2}}} \operatorname{prox}_{\gamma_{2}\psi}(\cdot,\gamma_{2}\operatorname{div} z + u)&&\quad\text{ a.e.\ in }\Omega\,.
\end{alignat}\\[-4.5mm]\notag
\end{subequations}
In summary, from \eqref{subsec:dual-invariance-inclusions.5} in \eqref{subsec:dual-invariance-inclusions.4}, we conclude that\vspace{-0.5mm}
\begin{align*}
  \mathtt{F}^{*}(z,u) = (\tfrac{1}{\gamma_1},\tfrac{1}{\gamma_2}) \mathtt{F}(z,u)\quad\text{ a.e.\ in }\Omega\,,\\[-6mm]\notag 
\end{align*}
\textit{i.e.}, the approaches based on the primal and dual formulations are equivalent up to the rescaling with $\sigma_{i} = \frac{1}{\gamma_{i}}$, $i=1,2$.  In other words, independent of whether one starts from the primal or the dual problem, up to rescaling, one ends up with the same $\operatorname{prox}$-based~\mbox{semi-smooth}~\mbox{Newton}~method, 
reflecting the usual symmetry between the primal and the dual problem in the case of strong duality. In particular, in the case $\gamma_1=\gamma_2=1$, we have that $\mathtt{F}^{*}(z,u) = \mathtt{F}(z,u)$ a.e.\ in $\Omega$,  
so that the resulting $\operatorname{prox}$-based semi-smooth Newton methods coincide. In the numerical tests, in particular, in Test \hyperlink{test 1}{1}, it is indicated that choosing the proximity parameters $\gamma_1,\gamma_2>0$ too small or too large may lead to larger iteration counts or even failure of the $\operatorname{prox}$-based semi-smooth~Newton~method.\linebreak
The dual
invariance further supports the choice of proximity parameters $\gamma_1=\gamma_2=\mathcal{O}(1)$.\pagebreak

\section{The finite-dimensional setting}\label{sec:discrete}

\subsection{Discrete regularized TV-minimization problem}

\hspace{5mm} Given a family of shape-regular triangulations $\{\mathcal{T}_h\}_{h>0}$ of 
    $\Omega\subseteq \mathbb{R}^d$, $d\in\mathbb{N}$, (\textit{cf}.\  {\cite[Def.\ 3.4]{Bartels2015}}), where $h\coloneqq \max_{T\in \mathcal{T}_h}{\{h_T\coloneqq \textup{diam}(T)\}}>0$ refers to the \emph{maximal mesh-size}, we let\vspace{-0.5mm}
\begin{align*}
  V_{h} \coloneqq \mathcal{S}^{1}_{D}(\mathcal{T}_{h})\,, 
  \qquad
  Y_{h} \coloneqq (\mathcal{L}^{0}(\mathcal{T}_{h}))^{d}\,,\\[-6mm]\notag
\end{align*}
denote the subordinated spaces of globally continuous, element-wise affine functions with vanishing traces on a Dirichlet boundary part $\Gamma_{D}\subseteq \partial\Omega$ and element-wise constant vector~fields,~respectively,
each equipped with the $L^2$-inner product.

Then, the \emph{discrete regularized TV-minimization problem} is defined as the minimization of the discrete primal energy functional $I_{h}^{\varepsilon}\colon V_{h}\to \mathbb{R}$, for every $v_{h}\in V_{h}$ defined by\vspace{-0.5mm}
\begin{align}\label{eq:discrete_primal}
  I_{h}^{\varepsilon}(v_{h})
  = \int_{\Omega}{|\nabla v_{h}|_{\varepsilon}\,\mathrm{d}x}
   + \frac{\alpha}{2}\int_{\Omega}{\vert v_{h}-g_{h}\vert^{2}\,\mathrm{d}x}\,,\\[-6mm]\notag
\end{align}
where $g_{h}\in V_{h}$ denotes the global $L^2$-projection of $g\in L^2(\Omega)$ onto $V_{h}$. By the direct method~in~the calculus of variations the discrete primal energy functional \eqref{eq:discrete_primal} admits~a~unique~minimizer~${u_{h}^{\varepsilon}\hspace{-0.1em}\in\hspace{-0.1em} V_{h}}$, called \emph{discrete primal solution}.

According to \cite[Prop.\ 4.4(i)]{BartelsKaltenbach2026}, a (Fenchel) dual problem (in the sense of \cite[Rem.\ 4.2,~p.~60/61]{EkelandTemam1999}) to the minimization of \eqref{eq:discrete_primal} is given via the maximization of the discrete dual energy functional $D_{h}^{\varepsilon}\colon Y_{h}\to \mathbb{R}\cup\{-\infty\}$, for every $y_h\in Y_{h}$ defined by\vspace{-0.5mm}
\begin{align}\label{eq:discrete_dual}
    D_{h}^{\varepsilon}(y_h)\coloneqq -\frac{\varepsilon}{2}\int_{\Omega}{\vert y_h\vert^2\,\mathrm{d}x}-I_{K_1(0)}^{\Omega}(y_h)-\frac{1}{2\alpha}\int_{\Omega}{\vert\!\operatorname{div}_h y_h+\alpha g_h\vert^2\,\mathrm{d}x}+\frac{\alpha}{2}\int_{\Omega}{\vert g_h\vert^2\,\mathrm{d}x}\,,\\[-6mm]\notag
\end{align}
where the \emph{discrete divergence operator} $\operatorname{div}_h \colon Y_{h}\to V_{h}$, for every $y_h\in Y_{h}$ and $v_h\in V_{h}$,~is~defined~by\vspace{-0.5mm}
\begin{align}\label{eq:discrete_divergence}
    \int_{\Omega}{\operatorname{div}_h y_hv_h\,\mathrm{d}x}=-\int_{\Omega}{y_h\cdot\nabla v_h\,\mathrm{d}x}\,.\\[-6mm]\notag
\end{align}
By \cite[Prop.\ 4.4(ii)]{BartelsKaltenbach2026}, there exists a maximizer $z_{h}^{\varepsilon}\in Y_{h}$ of the discrete dual energy~functional~\eqref{eq:discrete_dual}, called \emph{discrete dual solution}, discrete strong duality applies, \textit{i.e.}, we have that\vspace{-0.5mm}
\begin{align}\label{eq:discrete_strong_duality}
    I_{h}^{\varepsilon}(u_{h}^{\varepsilon})=D_{h}^{\varepsilon}(z_{h}^{\varepsilon})\,,\\[-6mm]\notag
\end{align}
and discrete optimality conditions apply, \textit{i.e.}, we have that\vspace{-0.5mm}
\begin{subequations}\label{eq:discrete_optimality}
    \begin{alignat}{2}\label{eq:discrete_optimality.3}
    \int_{\Omega}{\vert \nabla u_{h}^{\varepsilon}\vert_\varepsilon\,\mathrm{d}x}+\frac{\varepsilon}{2}\int_{\Omega}{\vert z_{h}^{\varepsilon}\vert^2\,\mathrm{d}x}
        &= -\int_{\Omega}{u_h^{\varepsilon}\operatorname{div}_h z_h^{\varepsilon}\,\mathrm{d}x}\,,\\
        \operatorname{div}_h  z_{h}^{\varepsilon}&=\alpha(u_{h}^{\varepsilon}-g_{h})\quad\text{ in }\Omega\,.\label{eq:discrete_optimality.2}
    \end{alignat}
\end{subequations}
By the equality condition in the Fenchel--Young inequality (\textit{cf}.\ \cite[Prop.~5.1, p.\ 21]{EkelandTemam1999}), the discrete optimality condition~\eqref{eq:discrete_optimality.3} is equivalent to\vspace{-0.5mm}
\begin{align}\label{eq:discrete_optimality.1}
     z_{h}^{\varepsilon} = \mathrm{D}|\cdot|_{\varepsilon}(\nabla u_{h}^{\varepsilon})\quad\text{ a.e.\ in }\Omega\,.\\[-6mm]\notag
\end{align}
By the definition of the discrete divergence operator \eqref{eq:discrete_divergence}, the discrete optimality condition~\eqref{eq:discrete_optimality.2} states that for every $v_{h}\in V_{h}$, there holds\vspace{-0.5mm}
\begin{align*}
    \int_{\Omega}{z_{h}^{\varepsilon}\cdot\nabla v_{h}\,\mathrm{d}x}+\alpha\int_{\Omega}{(u_{h}^{\varepsilon}-g)v_{h}\,\mathrm{d}x}  = 0 \,,\\[-6mm]\notag
\end{align*}
which shows that the global $L^2$-projection $g_h\in V_h$ 
in \eqref{eq:discrete_primal} is merely an auxiliary function that does not need to be computed in practice; instead, the original noisy image $g\in L^2(\Omega)$~can~be~used.

Multiplying the discrete optimality condition \eqref{eq:discrete_optimality.1} by a proximity parameter $\gamma>0$,
adding $\nabla u_{h}$ to both sides, and incorporating the proximity operator ${\operatorname{prox}_{\smash{\gamma\vert \cdot\vert_{\varepsilon}}}\hspace{-0.1em}=\hspace{-0.1em}(\operatorname{id}+\gamma\mathrm{D}|\cdot|_{\varepsilon})^{-1}\colon \hspace{-0.1em} \mathbb{R}^d\hspace{-0.1em}\to\hspace{-0.1em} \mathbb{R}^d}$, we may rewrite the discrete optimality condition \eqref{eq:discrete_optimality.1} as\vspace{-0.5mm}
\begin{align}
\nabla u_{h}^{\varepsilon} = \operatorname{prox}_{\gamma\vert \cdot\vert_{\varepsilon}}(\nabla u_{h}^{\varepsilon} + \gamma z_{h}^{\varepsilon})\quad\text{ a.e.\ in }\Omega\,.\label{eq:discrete_prox_equations}
\end{align}

As a consequence, the discrete optimality conditions \eqref{eq:discrete_optimality}  are equivalent to the  root finding problem for the nonlinear mapping $\mathtt{F}_{h}^{\varepsilon}\colon Y_{h}\times V_{h}\to Y_{h}\times V_{h}$, for every $(y_{h},v_{h})\in Y_{h}\times V_{h}$~defined~by
\begin{align}\label{def:F_h_eps}
    \mathtt{F}_{h}^{\varepsilon}(y_{h},v_{h}) \coloneqq
  \begin{bmatrix}
    \nabla v_{h} - \operatorname{prox}_{\gamma\vert \cdot\vert_{\varepsilon}}(\nabla v_{h} + \gamma y_{h})\\
    \alpha(v_{h}-g_{h}) - \operatorname{div}_{h} y_{h}
  \end{bmatrix}\quad\text{ a.e.\ in }\Omega\,,
\end{align}
which, by the equations \eqref{eq:discrete_prox_equations}, admits the root $(z_{h}^{\varepsilon},u_{h}^{\varepsilon})\in Y_{h}\times V_{h}$, \textit{i.e.}, we have that
\begin{align*}
    \mathtt{F}_{h}^{\varepsilon}(z_{h}^{\varepsilon},u_{h}^{\varepsilon}) =(\mathtt{0}_d,0)\quad \text{ in }Y_{h}\times V_{h}\,.
\end{align*}

Moreover, since $\smash{\operatorname{prox}_{\gamma\vert \cdot\vert_{\varepsilon}}}\colon \mathbb{R}^d\to \mathbb{R}^d$ is globally Newton differentiable with possible (global) Newton derivative $\mathtt{J}_{\smash{\operatorname{prox}_{\gamma\vert \cdot\vert_{\varepsilon}}}}\colon \mathbb{R}^d\to \mathbb{R}^{d\times d}$, defined by \eqref{def:prox_newton}, it is readily seen that the mapping $\mathtt{F}_{h}^{\varepsilon}\colon Y_{h}\times V_{h}\to Y_{h}\times V_{h}$, defined by \eqref{def:F_h_eps}, is also globally Newton differentiable with possible (global) Newton derivative $ \mathtt{J}_{\mathtt{F}_{h}^{\varepsilon}}\colon Y_{h}\times V_{h}\to \mathcal{L}(Y_{h}\times V_{h};Y_{h}\times V_{h})$, for every $(y_h,v_h),(\widehat{y}_h,\widehat{v}_h)\in Y_{h}\times V_{h}$ defined by
\begin{align}\label{def:J_F_h_eps}
  \mathtt{J}_{\mathtt{F}_{h}^{\varepsilon}}(y_{h},v_{h})(\widehat{y}_h,\widehat{v}_h) \coloneqq
  \begin{bmatrix}
    (\mathbbone - \mathtt{J}_{\operatorname{prox}_{\gamma\vert \cdot\vert_{\varepsilon}}}(a_h) )\nabla \widehat{v}_h-\gamma \mathtt{J}_{\operatorname{prox}_{\gamma\vert \cdot\vert_{\varepsilon}}}(a_h)\widehat{y}_h  \\
    \alpha\widehat{v}_h- \operatorname{div}_{h}\widehat{y}_h
  \end{bmatrix}\quad \text{ a.e.\ in }\Omega\,,
\end{align}
where $a_h\hspace{-0.1em}\coloneqq\hspace{-0.1em} \nabla v_{h} + \gamma y_{h}\hspace{-0.1em}\in\hspace{-0.1em} Y_h$, which gives rise to the following discrete semi-smooth Newton~method:

\begin{algorithm}[Semi-smooth Newton method for discrete, regularized~\mbox{TV-minimization}]\label{alg:SSNM-disc}
Let $\varepsilon>0$ be a regularization parameter, let $\gamma>0$ be a proximity parameter, let 
$\varepsilon^{\mathtt{stop}}_{\mathtt{loc},h}>0$~be~a~\mbox{stopping} parameter, let $u^{0}_{h}\in V_{h}$ be an initial iterate, set $z^{0}_{h} \hspace{-0.1em}\coloneqq \hspace{-0.1em}\mathrm{D}|\cdot|_{\varepsilon}(\nabla u^{0}_{h})\hspace{-0.1em}\in\hspace{-0.1em} Y_{h}$, and let $k^{\mathtt{max}}_{\mathtt{loc},h}\hspace{-0.1em}\in \hspace{-0.1em}\mathbb{N}\cup\{+\infty\}$~be the \hspace{-0.1mm}number \hspace{-0.1mm}of \hspace{-0.1mm}maximum \hspace{-0.1mm}iterations. \hspace{-0.1mm}Then, \hspace{-0.1mm}for \hspace{-0.1mm}$k=0,\ldots,k^{\mathtt{max}}_{\mathtt{loc},h}$,~\hspace{-0.1mm}\mbox{perform}~\hspace{-0.1mm}the~\hspace{-0.1mm}\mbox{following}~\hspace{-0.1mm}\mbox{iteration}~\hspace{-0.1mm}loop: 
\begin{itemize}[noitemsep,topsep=2pt,leftmargin=!,labelwidth=\widthof{(2)}]
\item[(1)]\hypertarget{alg:SSNM-disc.1}{} Compute the primal-dual update direction $(\delta z_{h}^k,\delta u_{h}^k)\in Y_{h}\times V_{h}$ such that
\begin{align}\label{alg:SSNM-disc_formula}
  \mathtt{J}_{\mathtt{F}_{h}^{\varepsilon}}(z^{k}_{h},u^{k}_{h})(\delta z_{h}^k,\delta u_{h}^k) = -\mathtt{F}_{h}^{\varepsilon}(z^{k}_{h},u^{k}_{h})\quad\text{ in }Y_{h}\times V_{h}\,,
\end{align}
and the updated iterate $(z^{k+1}_{h},u^{k+1}_{h})
  \coloneqq (z^{k}_{h},u^{k}_{h}) + (\delta z_{h}^{k},\delta u_{h}^{k})\in Y_{h}\times V_{h}$; 
\item[(2)]\hypertarget{alg:SSNM-disc.2}{} If $\|\mathtt{F}_{h}^{\varepsilon}(z_{h}^{k+1},u_{h}^{k+1})\|_{Y_{h}\times V_{h}}<\varepsilon^{\mathtt{stop}}_{\mathtt{loc},h}$, then \textup{STOP}; otherwise, $k\to k+1$ and continue~with~(\hyperlink{alg:SSNM-disc.1}{1}). 
\end{itemize}
\end{algorithm}

Algorithm \ref{alg:SSNM-disc} is globally well-posed and locally super-linearly convergent.

\begin{theorem}\label{theorem:SSNM-disc}
The following statements apply:
\begin{itemize}[noitemsep,topsep=2pt,leftmargin=!,labelwidth=\widthof{(ii)}]
    \item[(i)] \hypertarget{theorem:SSNM-disc.i}{} Algorithm \ref{alg:SSNM-disc} is globally well-posed, \textit{i.e.}, if $k^{\mathtt{max}}_{\mathtt{loc},h}=+\infty$, given an arbitrary initial iterate $ u_{h}^0\in V_{h}$, all iterates $\{(z^{k}_{h},u^{k}_{h})\}_{k\in \mathbb{N}}\subseteq Y_{h}\times V_{h} $ are computable;
    
    \item[(ii)] \hypertarget{theorem:SSNM-disc.ii}{} Algorithm \ref{alg:SSNM-disc} is locally super-linearly convergent, \textit{i.e.}, if the initial iterate $u_{h}^0\in V_{h}$~is~sufficiently close to $u_{h}^{\varepsilon}\in V_{h}$, then for every $\eta>0$, there exists some $K\in \mathbb{N}$ such that for every $k\in \mathbb{N}$ with $k\ge K$, there holds
    \begin{align*}
        \|(z^{k+1}_{h},u^{k+1}_{h})-(z^{\varepsilon}_{h},u^{\varepsilon}_{h})\|_{Y_{h}\times V_{h}}\leq \eta\,\|(z^{k}_{h},u^{k}_{h})-(z^{\varepsilon}_{h},u^{\varepsilon}_{h})\|_{Y_{h}\times V_{h}}\,.
    \end{align*}
\end{itemize}
\end{theorem}

\begin{proof}
    \emph{ad (\hyperlink{theorem:SSNM-disc.i}{i}).} It suffices to establish that for every $(y_{h},v_{h})\in Y_{h}\times V_{h}$, the Newton derivative $\mathtt{J}_{\mathtt{F}_{h}^{\varepsilon}}(y_{h},v_{h})\in \mathcal{L}(Y_{h}\times V_{h};Y_{h}\times V_{h})$, defined by \eqref{def:J_F_h_eps}, is invertible. To this end, let $(y_{h},v_{h})\in Y_{h}\times V_{h}$ as well as $(f_{h}^1,f_{h}^2)\in Y_{h}\times V_{h}$ be fixed, but arbitrary. Then, we seek $(\delta y_{h}, \delta v_{h})\in Y_{h}\times V_{h}$~such~that
    \begin{align}\label{theorem:SSNM-disc.1}
        \mathtt{J}_{\mathtt{F}_{h}^{\varepsilon}}(y_{h},v_{h})(\delta y_{h},\delta v_{h})=(f_{h}^1,f_{h}^2)\quad\text{ in }Y_{h}\times V_{h}\,.
    \end{align}
    First, we note that  the system \eqref{theorem:SSNM-disc.1} reads 
    \begin{subequations} \label{theorem:SSNM-disc.2}
    \begin{alignat}{2}\label{theorem:SSNM-disc.2.1}
        \delta y_{h}&=\tfrac{1}{\gamma}\mathtt{J}_h^{-1}\{(\mathbbone - \mathtt{J}_h) \nabla \delta v_{h}+f_{h}^1\}&&\quad\text{ in }Y_h\,,\\
        -\operatorname{div}_h  \delta y_{h}+\alpha \delta v_{h}&=f_{h}^2&&\quad\text{ in }V_h\,,\label{theorem:SSNM-disc.2.2}
    \end{alignat}
    \end{subequations}
    where (again, abbreviating $a_h\coloneqq \nabla v_h+\gamma y_h\in Y_{h}$ in \eqref{def:J_F_h_eps}) the invertibility of $\mathtt{J}_h\coloneqq  \smash{\mathtt{J}_{\smash{\operatorname{prox}_{\gamma\vert\cdot\vert_{\varepsilon}}}}(a_h)}\in  (L^\infty(\Omega))^{d\times d}$ is~based~on \eqref{eq:J_prox_lower_upper_improved}.  
    Next, inserting \eqref{theorem:SSNM-disc.2.1} in \eqref{theorem:SSNM-disc.2.2} and rearranging, we arrive at the discrete linear elliptic problem
    \begin{align}
        -\operatorname{div}_h \bigl(\tfrac{1}{\gamma}\mathtt{J}_h^{-1}(\mathbbone - \mathtt{J}_h)\nabla\delta v_h\bigr)+\alpha \delta v_{h}=-\operatorname{div}_h \bigl(\tfrac{1}{\gamma}\mathtt{J}_h^{-1}f_h^1)+f_h^2\quad\text{ in }V_h\,.
    \end{align}
    Let the linear operators 
    $A_h\colon V_h\to V_h^*$ and 
    $\ell_h\colon V_h\to \mathbb{R}$, for every $\widehat{v}_h,\widehat{w}_h\in V_h$, be defined by\vspace{-0.5mm}
    \begin{subequations} \label{theorem:SSNM-disc.5}
    \begin{align}
        \langle A_h\widehat{v}_h,\widehat{w}_h\rangle_{V_h}&\coloneqq \frac{1}{\gamma}\int_{\Omega}{\mathtt{J}_h^{-1}(\mathbbone-\mathtt{J}_h)\nabla \widehat{v}_{h}\cdot \nabla \widehat{w}_{h}\,\mathrm{d}x}+\alpha\int_{\Omega}{\widehat{v}_{h}\widehat{w}_{h}\,\mathrm{d}x}\,,\\
        \langle \ell_h,\widehat{w}_{h}\rangle&\coloneqq \frac{1}{\gamma}\int_{\Omega}{\mathtt{J}_h^{-1}f_{h}^1\cdot\nabla \widehat{w}_{h}\,\mathrm{d}x}+\int_{\Omega}{f_{h}^2\widehat{w}_h\,\mathrm{d}x}\,.
    \end{align}\\[-6.5mm]\notag
    \end{subequations}
    
    On the one hand, since, due to \eqref{eq:J_prox_lower_upper} and \eqref{eq:J_prox_lower_upper_improved}, we have that $\mathtt{J}_h^{-1}(\mathbbone-\mathtt{J}_h)\succeq 0$ a.e.\ in $\Omega$, for every $\widehat{v}_h\in V_h$, we find that\vspace{-0.5mm}
    \begin{align}\label{theorem:SSNM-disc.6}
        \langle A_h\widehat{v}_h,\widehat{v}_h\rangle_{V_h}
        =
        \alpha \|\widehat{v}_h\|_{V_h}^2\,.\\[-6mm]\notag
    \end{align}

    On the other hand, by Hölder’s inequality and an inverse estimate (with constant $c_{\mathtt{inv,h}}>0$, \textit{cf}.\ \cite[Lem.~3.5]{Bartels2015}), we have that\vspace{-0.5mm}
    \begin{align}\label{theorem:SSNM-disc.7}
        \smash{\|\ell_h\|_{(V_h)^*}\leq \bigl(\tfrac{c_{\mathtt{inv,h}}}{\gamma}\|\mathtt{J}_h^{-1}\|_{\infty,\Omega}+1\bigr)\|(f_{h}^1,f_{h}^2)\|_{Y_h\times V_h}\,.}\\[-6mm]\notag
    \end{align}
    Then, by the Lax--Milgram lemma (\textit{cf}.\ \cite[Cor.\ 5.8]{Brezis2011}), there exists a unique $\delta v_h\in V_h$ such that\vspace{-0.5mm} 
    \begin{align}\label{theorem:SSNM-disc.8}
        \smash{A_h\delta v_{h}=\ell_h\quad\text{ in }(V_h)^* \,,} \\[-6mm]\notag
    \end{align}
    which, due to \eqref{theorem:SSNM-disc.6} and \eqref{theorem:SSNM-disc.7}, satisfies the \textit{a priori} estimate\vspace{-0.5mm}
    \begin{align}\label{theorem:SSNM-disc.9}
        \smash{\|\delta v_{h}\|_{V_h}\leq \tfrac{1}{\alpha }\bigl(\tfrac{c_{\mathtt{inv,h}}}{\gamma}\|\mathtt{J}_h^{-1}\|_{\infty,\Omega}+1\bigr)\|(f_{h}^1,f_{h}^2)\|_{Y_h\times V_h}\,.}\\[-6mm]\notag
    \end{align}
    Due to \eqref{theorem:SSNM-disc.8} and the definitions \eqref{theorem:SSNM-disc.5}, 
    the pair $(\delta y_{h},\delta v_{h})\in Y_{h}\times V_{h}$, where $\delta y_h\in Y_{h}$~is~defined~via  \eqref{theorem:SSNM-disc.2.1}, is the unique solution of the system~\eqref{theorem:SSNM-disc.1}  and, due to  \eqref{theorem:SSNM-disc.2.1},~\eqref{theorem:SSNM-disc.9},~and~\eqref{theorem:SSNM-disc.7},~satisfies\vspace{-0.5mm}
    \begin{align*}
        \begin{aligned} 
        \|(\delta y_h,\delta v_h)\|_{Y_h\times V_h}
        &\leq (\tfrac{1}{\gamma}\|\mathtt{J}_h^{-1}(\mathbbone-\mathtt{J}_h)\|_{\infty,\Omega}+1)\|\delta v_h\|_{V_h}+\tfrac{1}{\gamma}\|\mathtt{J}_h^{-1}\|_{\infty,\Omega}\|f_h^1\|_{Y_h}
        \\&\leq \smash{(\tfrac{1}{\gamma}\|\mathtt{J}_h^{-1}(\mathbbone-\mathtt{J}_h)\|_{\infty,\Omega}+1)\tfrac{1}{\alpha }\bigl(\tfrac{c_{\mathtt{inv,h}}}{\gamma}\|\mathtt{J}_h^{-1}\|_{\infty,\Omega}+1\bigr)\|(f_{h}^1,f_{h}^2)\|_{Y_h\times V_h}}\\&\quad+\smash{\tfrac{1}{\gamma}\|\mathtt{J}_h^{-1}\|_{\infty,\Omega}\|f_h^1\|_{Y_h}}
        \\&\leq  C_h\|(f_{h}^1,f_{h}^2)\|_{Y_h\times V_h}\,,
        \end{aligned}\\[-6.5mm]\notag
    \end{align*}
    where $C_h\coloneqq (\tfrac{1}{\gamma}\|\mathtt{J}_h^{-1}(\mathbbone-\mathtt{J}_h)\|_{\infty,\Omega}+1)\tfrac{1}{\alpha }(\tfrac{c_{\mathtt{inv,h}}}{\gamma}\|\mathtt{J}_h^{-1}\|_{\infty,\Omega}+1)+\tfrac{1}{\gamma}\|\mathtt{J}_h^{-1}\|_{\infty,\Omega}>0$. 
    In other words, the Newton derivative $\smash{\mathtt{J}_{\mathtt{F}_{h}^{\varepsilon}}(y_{h},v_{h})}\in \mathcal{L}(Y_{h}\times V_{h};Y_{h}\times V_{h})$ is invertible with\vspace{-0.5mm}
    \begin{align}\label{theorem:SSNM-disc.12}
        \smash{\|\mathtt{J}_{\mathtt{F}_{h}^{\varepsilon}}(y_{h},v_{h})^{-1}\|_{\mathcal{L}(Y_{h}\times V_{h};Y_{h}\times V_{h})}\leq C_{h}\,.}\\[-6mm]\notag
    \end{align}
    In particular, we conclude the claimed global well-posedness of Algorithm \ref{alg:SSNM-disc}.

    \emph{ad (\hyperlink{theorem:SSNM-disc.ii}{ii}).} Since, according to the proof of claim (\hyperlink{theorem:SSNM-disc.i}{i}),  the Newton derivative \eqref{def:J_F_h_eps} is point-wise invertible with uniformly bounded inverse (\textit{cf}.\ \eqref{theorem:SSNM-disc.12}), resorting to standard results on local super-linear convergence of semi-smooth Newton methods (see, \textit{e.g.}, \cite[Thm.~3.4]{ChenNashedQi2000}, \cite[Thm.~1.1]{HintermüllerItoKunisch2003},~\cite[Thm.~3.13]{Ulbrich2003}), we conclude the claimed local super-linear convergence of Algorithm~\ref{alg:SSNM-disc}.
\end{proof}

\begin{remark}[Reduced formulation of update direction computation]\label{rem:reduced_problem}
    The perturbed saddle point problem \eqref{alg:SSNM-disc_formula}, which computes the primal-dual update direction $(\delta z_{h}^{k},\delta u_{h}^{k})\in Y_{h}\times V_{h}$ from the most recent iterate $(z_{h}^{k},u_{h}^{k})\in Y_{h}\times V_{h}$ can be solved in two steps via a linear elliptic problem:
    \begin{itemize}[noitemsep,topsep=2pt,leftmargin=!,labelwidth=\widthof{(2)}]
        \item[(1)] Compute the \emph{primal update direction} $\delta u_h^{k}\in V_h$~such~that for every $v_h\in V_h$,~there~holds\vspace{-0.5mm}
    \begin{align}\label{rem:reduced_problem.1}
            &\frac{1}{\gamma}\int_{\Omega}{(\mathtt{J}_h^k)^{-1}(\mathbbone-\mathtt{J}_h^k)\nabla \delta u_{h}^{k}\cdot \nabla v_{h}\,\mathrm{d}x}+\alpha\int_{\Omega}{\delta u_{h}^{k}v_{h}\,\mathrm{d}x} 
        \\&=
        -\frac{1}{\gamma}\int_{\Omega}{(\mathtt{J}_h^k)^{-1}(\nabla u_h^k-\operatorname{prox}_{\gamma\vert \cdot\vert_{\varepsilon}}(a_{h}^{k}))\cdot \nabla v_h\,\mathrm{d}x}-\alpha \int_{\Omega}{(u_{h}^k-g_{h})v_{h}\,\mathrm{d}x}-\int_{\Omega}{z_h^k\cdot \nabla v_h\,\mathrm{d}x}\,,\notag\\[-6mm]\notag
    \end{align}
    where $\smash{\mathtt{J}_h^k\coloneqq\mathtt{J}_{\operatorname{prox}_{\gamma\vert \cdot\vert_{\varepsilon}}}(a_{h}^{k})\in (L^\infty(\Omega))^{d\times d}}$ and 
    $\smash{a_h^k\coloneqq \nabla u_h^k+\gamma z_h^k\in Y_h}$.
    Since the linear problem \eqref{rem:reduced_problem.1} is symmetric positive definite, it may be solved efficiently by sparse Cholesky factorization for moderate problem size and a preconditioned conjugate~gradient~method~\mbox{otherwise}; 
    \item[(2)] Compute the \emph{dual update direction} $\delta z_{h}^{k}\in Y_{h}$, given via the explicit representation formula\vspace{-0.5mm} 
    \begin{align}\label{rem:reduced_problem.2}
       \smash{\delta z_{h}^{k}\coloneqq \tfrac{1}{\gamma}(\mathtt{J}_{\operatorname{prox}_{\gamma\vert \cdot\vert_{\varepsilon}}}(a_{h}^{k}))^{-1}\bigl((\mathbbone - \mathtt{J}_{\operatorname{prox}_{\gamma\vert \cdot\vert_{\varepsilon}}}(a_{h}^{k})) \nabla \delta u_{h}^{k}+(\nabla u_{h}^{k}-\operatorname{prox}_{\gamma\vert \cdot\vert_{\varepsilon}}(a_{h}^{k}))\bigr)\in Y_h\,.}\\[-6mm]\notag
   \end{align}
    \end{itemize}

\end{remark}\newpage

A globalization approach employs a semi-implicit discretization of a discrete $L^{2}$-gradient~flow.\enlargethispage{5mm}\vspace{-0.5mm} 

\begin{algorithm}[Semi-implicit discretized $L^2$-gradient flow for discrete, regularized~TV-minimiza\-tion]\label{alg:gradflow-disc}
Let $\varepsilon\hspace{-0.15em}>\hspace{-0.15em}0$ be a regularization~parameter,  let $\varepsilon^{\mathtt{stop}}_{\mathtt{glob},h}\hspace{-0.15em}>\hspace{-0.15em}0$ be  a stopping parameter, let $u^{0}_{h}\in V_{h}$ be an initial iterate, and let 
$ k^{\mathtt{max}}_{\mathtt{glob},h}\in \mathbb{N}\cup\{+\infty\}$ be a number of maximum iterations. Then, for  $k=0,\ldots,\smash{k^{\mathtt{max}}_{\mathtt{glob},h}}$, perform the following iteration loop:
\begin{itemize}[noitemsep,topsep=2pt,leftmargin=!,labelwidth=\widthof{(2)}]
\item[(1)] \hypertarget{alg:gradflow-disc.1}{} Compute the updated iterate $u^{k+1}_{h}\in V_{h}$ such that for every $v_h\in V_{h}$, there holds\vspace{-0.5mm}
\begin{align*}
\int_{\Omega}{\mathrm{d}_{\tau}u^{k+1}_{h}v_{h}\,\mathrm{d}x}+
  \int_{\Omega}{\smash{\frac{\phi_{\varepsilon}'(|\nabla u^{k}_{h}|)}{|\nabla u^{k}_{h}|}}
  \nabla u^{k+1}_{h}\cdot\nabla v_{h}\,\mathrm{d}x}
    +\alpha\int_{\Omega}{(u^{k+1}_{h}-g)v_{h}\,\mathrm{d}x} = 0\,;\\[-6mm]\notag
\end{align*}   

\item[(2)] \hypertarget{alg:gradflow-disc.2}{} If $\|\mathtt{F}_{h}^{\varepsilon}(z_{h}^{k+1},u_{h}^{k+1})\|_{Y_{h}\times V_{h}}\hspace{-0.15em}< \hspace{-0.15em}\varepsilon^{\mathtt{stop}}_{\mathtt{glob},h}$, where $z_{h}^{k+1} \hspace{-0.15em}\coloneqq\hspace{-0.15em} \phi_{\varepsilon}'(|\nabla u^{k}_{h}|)|\nabla u^{k}_{h}|^{-1}\nabla u^{k+1}_{h}\hspace{-0.15em}\in\hspace{-0.15em} Y_{
h}$,~then~\textup{STOP};  otherwise, $k\to k+1$ and continue with (\hyperlink{alg:gradflow-disc.1}{1}).
\end{itemize}
\end{algorithm}

Algorithm \ref{alg:gradflow-disc} is globally well-posed, unconditionally strongly stable, and terminates~after~a finite number of iterations.\vspace{-0.5mm}

\begin{proposition}\label{prop:strong_stability}
Algorithm \ref{alg:gradflow-disc} is globally well-posed, \textit{i.e.}, if $k^{\mathtt{max}}_{\mathtt{glob},h}=+\infty$, 
given an arbitrary initial iterate $u_{h}^{0}\in V_{h}$, all iterates $\{u^{k}_{h}\}_{k\in \mathbb{N}}\subseteq V_{h}$ are computable, unconditionally strongly stable, \textit{i.e.}, for every $K\in \mathbb{N}$, we have that\vspace{-0.5mm}
\begin{align}\label{prop:strong_stability.0}
  I_{h}^{\varepsilon}(u^{K}_{h})
  + \tau\sum_{k=1}^{K}\|\mathrm{d}_{\tau}u^{k}_{h}\|_{V_h}^{2}
  \le I_{h}^{\varepsilon}(u^{0}_{h})\,,\\[-6mm]\notag
\end{align}
and terminates after a finite number of iterations, \textit{i.e.}, if $\smash{k^{\mathtt{max}}_{\mathtt{glob},h}}=+\infty$, then there exists  $k^*\in \mathbb{N}$ such that $\|\mathtt{F}_{h}^{\varepsilon}(z^{k^*+1}_{h},u^{k^*+1}_{h})\|_{Y_{h}\times V_{h}}<\smash{\varepsilon^{\mathtt{stop}}_{\mathtt{glob},h}}$.
\end{proposition}

\begin{proof} For the well-posedness and unconditional strong stability, we refer to 
\cite[Prop.\ 3.4]{BartelsDieningNochetto2018}.

By analogy to the proof of   \cite[Prop.\ 5.2(iii)]{BartelsKaltenbach2024}, we find that $u^{k}_{h}\to u_h^{\varepsilon}$ in $V_{h}$ $(k\to \infty)$, 
which, by the continuity of $(t\mapsto \phi_{\varepsilon}'(|t|)\vert t\vert^{-1})\colon \mathbb{R}^d\to \mathbb{R}^d$, implies that $z_{h}^k\to z_h^{\varepsilon}$ in $Y_{h}$ $(k\to \infty)$.~By the continuity of $\mathtt{F}_{h}^{\varepsilon}\colon Y_{h}\times V_{h}\to Y_{h}\times V_{h}$, 
we infer that $\mathtt{F}_{h}^{\varepsilon}(z_{h}^k,u^{k}_{h})\to \mathtt{F}_{h}^{\varepsilon}(z_{h}^{\varepsilon},u_{h}^{\varepsilon})=(\mathtt{0}_d,0)$~a.e.~in~$\Omega$~${(k\to \infty)}$. In particular, there exists some $k^*\in \mathbb{N}$ such that $\|\mathtt{F}_{h}^{\varepsilon}(z^{k^*+1}_{h},u^{k^*+1}_{h})\|_{Y_{h}\times V_{h}}<\varepsilon^{\mathtt{stop}}_{\mathtt{glob},h}$.
\end{proof}

 Inspired by \cite{BartelsGudiKaltenbach2025,AntilBartelsKaltenbachKhandelwal2025,DieningStorn2025},
    instead of the nonlinear residual norm $\|\mathtt{F}_h(\cdot,\cdot)\|_{Y_h\times V_h}\colon Y_h\times V_h\to \mathbb{R}$, one can employ a duality-based residual norm given via the \emph{discrete primal-dual gap estimator} $\eta_{\mathtt{gap},h}^2\colon V_h\times Y_h\to \mathbb{R}\cup\{+\infty\}$, for every $(v_h,y_h)\in V_h\times Y_h $ defined by\vspace{-1mm}
    \begin{align}\label{def:eta_gap}
        \smash{\eta_{\mathtt{gap},h}^2(v_h,y_h)\coloneqq I_h^{\varepsilon}(v_h)-D_h^{\varepsilon}(y_h)\,,}\\[-5.5mm]\notag
    \end{align}
    which, by the discrete strong duality relation $I_{h}^{\varepsilon}(u_{h}^{\varepsilon})=D_{h}^{\varepsilon}(z_{h}^{\varepsilon})$,  explicitly controls the sum of the optimal strong convexity measures $\rho_{I_h^{\varepsilon}}^2\colon V_h\to \mathbb{R}$ and $\rho_{-D_h^{\varepsilon}}^2\colon Y_h\to \mathbb{R}\cup\{+\infty\}$  of \eqref{eq:discrete_primal} and \eqref{eq:discrete_dual} (recall that \eqref{eq:discrete_dual} is concave), for every $v_h\in V_h$ and $y_h\in Y_h$, respectively, defined by\vspace{-0.5mm}
    \begin{align}\label{def:rho}
        \smash{\rho_{I_h^{\varepsilon}}^2(v_h)\coloneqq I_h^{\varepsilon}(v_h)-I_h^{\varepsilon}(u_h^{\varepsilon})\,,\qquad
        \rho_{-D_h^{\varepsilon}}^2(y_h)\coloneqq -D_h^{\varepsilon}(y_h)+D_h^{\varepsilon}(z_h^{\varepsilon})\,.}\\[-5.5mm]\notag
    \end{align}

\begin{theorem}[Discrete primal-dual gap identity]\label{thm:discrete_primal-dual_gap_identity}
   For every $(v_h,y_h)\in V_h\times Y_h $, there holds\vspace{-0.5mm}
   \begin{align}\label{thm:discrete_primal-dual_gap_identity.1}
       \smash{\rho_{I_h^{\varepsilon}}^2(v_h)+\rho_{-D_h^{\varepsilon}}^2(y_h)= \eta_{\mathtt{gap},h}^2(v_h,y_h)\,,}\\[-5.5mm]\notag
   \end{align}
   where, for every $v_h\in V_h$ and $y_h\in Y_h$ with $\vert y_h\vert\leq 1$ a.e.\ in $\Omega$, we have that\vspace{-0.5mm} 
   \begin{subequations}\label{thm:discrete_primal-dual_gap_identity.2}
   \begin{align}\label{thm:discrete_primal-dual_gap_identity.2.1}
           \rho_{I_h^{\varepsilon}}^2(v_h)&=\int_{\Omega}{\bigl(\vert \nabla v_h\vert_{\varepsilon}\hspace{-0.1em}-\hspace{-0.1em}\nabla v_h\cdot z_h^{\varepsilon}+\tfrac{\varepsilon}{2}\vert z_{h}^{\varepsilon}\vert^2\bigr)\,\mathrm{d}x}\hspace{-0.1em}+\hspace{-0.1em}\frac{\alpha}{2}\int_{\Omega}{\vert v_h\hspace{-0.1em}-\hspace{-0.1em}u_h^{\varepsilon}\vert^2\,\mathrm{d}x}\,;\\\label{thm:discrete_primal-dual_gap_identity.2.2}
           \rho_{-D_h^{\varepsilon}}^2(y_h)&=\int_{\Omega}{\bigl(\vert \nabla u_h^{\varepsilon}\vert_{\varepsilon}\hspace{-0.1em}-\hspace{-0.1em}\nabla u_h^{\varepsilon}\cdot y_h\hspace{-0.1em}+\hspace{-0.1em}\tfrac{\varepsilon}{2}\vert y_h\vert^2\bigr)\,\mathrm{d}x}\hspace{-0.1em}+\hspace{-0.1em}\frac{1}{2\alpha}\int_{\Omega}{\vert\!\operatorname{div}_h (y_h\hspace{-0.1em}-\hspace{-0.1em}z_h^{\varepsilon})\vert^2\,\mathrm{d}x}\,;\\\label{thm:discrete_primal-dual_gap_identity.2.3}
           \eta_{\mathtt{gap},h}^2(v_h,y_h)&=\int_{\Omega}{\bigl(\vert \nabla v_h\vert_{\varepsilon}
           \hspace{-0.1em}-\hspace{-0.1em}\nabla v_h\cdot y_h\hspace{-0.1em}+\hspace{-0.1em}\tfrac{\varepsilon}{2}\vert y_h\vert^2\bigr)\,\mathrm{d}x}\hspace{-0.1em}+\hspace{-0.1em}\frac{1}{2\alpha}\int_{\Omega}{\vert\!\operatorname{div}_h y_h\hspace{-0.1em}-\hspace{-0.1em}\alpha(v_h\hspace{-0.1em}-\hspace{-0.1em}g_h)\vert^2\,\mathrm{d}x}\,.
       \end{align}
    \end{subequations}
\end{theorem}

\begin{proof}
   \textit{ad \eqref{thm:discrete_primal-dual_gap_identity.1}.} The discrete primal-dual gap identity, due to the definitions \eqref{def:eta_gap} and \eqref{def:rho}, is a direct consequence of the discrete strong duality relation $I_{h}^{\varepsilon}(u_{h}^{\varepsilon})=D_{h}^{\varepsilon}(z_{h}^{\varepsilon})$.

   \textit{ad \eqref{thm:discrete_primal-dual_gap_identity.2}.} The proof of the integral representations in \eqref{thm:discrete_primal-dual_gap_identity.2} is postponed to the appendix.
\end{proof}\newpage

\begin{remark}
    By the Fenchel--Young inequality (\textit{cf}.\ \cite[Prop.~5.1, p.\ 21]{EkelandTemam1999}), for every $v_h\in V_h$ and $y_h\in Y_h$ with $\vert y_h\vert\leq 1$ a.e.\ in $\Omega$, we have that $\vert \nabla v_h\vert_{\varepsilon}-\nabla v_h\cdot y_h+\tfrac{\varepsilon}{2}\vert y_h\vert^2\ge 0$ a.e.\ in $\Omega$ with equality if and only if $y_h = \mathrm{D}|\cdot|_{\varepsilon}(\nabla v_h)$ a.e.\ in $\Omega$.\vspace{-1mm}
\end{remark}

\section{Numerical experiments}\label{sec:numerics}\vspace{-1mm}

\hspace{5mm}In this section, we investigate the developed $\operatorname{prox}$-based semi-smooth Newton method for the discrete, regularized TV-minimization problem (\textit{cf}.\ Algorithm~\ref{alg:SSNM-disc}) for its performance, considering the following  benchmark example for the TV-minimization problem that provides a discontinuous primal solution and a Lipschitz continuous dual solution (see, \textit{e.g.}, \cite[Expl.\ 10.4]{Bartels2015}).\vspace{-0.5mm}

\begin{example}\label{ex:41}
Let $\Omega = (-1,1)^{d}$, $d\in \mathbb{N}$, $\Gamma_{D} = \partial\Omega$, $\alpha>0$, $0<r<1$, and $g = \chi_{B_{r}^d(0)}\in L^2(\Omega)$.
Then, the unique primal solution $u\in BV(\Omega)\cap L^2(\Omega)$, \textit{i.e.}, the minimizer of  \eqref{intro:primal}, 
and a dual solution $z\in W^2_0(\operatorname{div};\Omega)$, \textit{i.e.}, a maximizer of \eqref{eq:rof_dual},  are given via\vspace{-0.5mm}
\begin{align}
\begin{aligned}
  u &\coloneqq \max\bigl\{ 0, 1 - \tfrac{d}{\alpha r}\bigr\}\chi_{B_{r}^d(0)}&&\quad \text{ a.e.\ in }\Omega\,,\\
  z &\coloneqq  
      -\smash{\tfrac{r}{\max\{r,\vert \cdot\vert\}^2}}\min\left\{1,\tfrac{\alpha r}{d}\right\}\mathrm{id}_{\mathbb{R}^d} &&\quad\text{ a.e.\ in }\Omega\,.
    \end{aligned}
\end{align} 
\end{example}

In the following, we carry out different numerical tests to experimentally investigate the role of the  parameters $\varepsilon,\gamma>0$ and to understand the choice of stopping criteria for the outer iteration. 
In doing so, we always consider the setup $d=2$, $r=\frac{1}{2}$, and $\alpha = 10$ in the following.
Moreover, unless otherwise specified, the numerical tests are run on sequences of triangulations~$\{\mathcal{T}_{h_i}\}_{i\in \mathbb{N}_0}$, 
obtained by means of red-refinement  starting from an initial triangulation $\mathcal{T}_0$ constructed by subdividing the square $\Omega$ along its
diagonal connecting $(-1,-1)^\top$ and $(1,1)^\top$~into~two~triangles.

All experiments were conducted using the finite element software package \texttt{FEniCS}  (version 2019.1.0, \textit{cf}.\  \cite{LoggMardalWells2012}). All graphics were generated using the \texttt{Matplotlib}~library~(version~3.5.1,~\textit{cf}.~\cite{Hunter2007}).

The linear systems \eqref{alg:SSNM-disc_formula} arising in the semi-smooth Newton iteration in Algorithm~\ref{alg:SSNM-disc} are solved according to the procedure outlined in Remark \ref{rem:reduced_problem}, where the symmetric, positive definite linear~system \eqref{rem:reduced_problem.1} for the computation of the primal update direction
is solved using \texttt{PETSc}'s (version 3.17.4, \textit{cf}.\ \cite{petsc-user-ref}) conjugate gradient method 
preconditioned~by~\texttt{HYPRE}~AMG, using a relative tolerance of $\tau_{\mathtt{rel}}\hspace{-0.15em}=\hspace{-0.15em}1\mathrm{e}{-}{10}$, an absolute tolerance of $\tau_{\mathtt{abs}}\hspace{-0.15em}=\hspace{-0.15em}1\mathrm{e}{-}{14}$,~and~a~\mbox{maximum}~of~$1000$~\mbox{iterations}.\vspace{-1mm}

\subsection*{Test 1: Role of regularization parameter $\varepsilon$}\hypertarget{test 1}{}\vspace{-0.5mm}

\hspace{5mm}In this test, we restrict to the triangulation $\mathcal{T}_{h_7}$, 
fix the proximity parameter $\gamma=1$, and consider different choices for the regularization parameter $\varepsilon = h_7^{\beta}$ for $\beta \in \{0.125,0.25,0.5,1,1.5,2,2.5,3\}$. The  \textit{a priori} error analyses in \cite{ChambollePock2020,Bartels2021,BartelsKaltenbach2022} demonstrate that the choice $\varepsilon \sim h$ is sufficient to obtain quasi-optimal error decay rates for the approximation of the TV-minimization problem~employing element-wise affine functions and uniform mesh-refinement. In this context, we globalize the $\operatorname{prox}$-based semi-smooth Newton method (\textit{cf}.\ Algorithm \ref{alg:SSNM-disc}) either by combining it with a globally convergent semi-implicit gradient flow method or by incorporating~an~\mbox{Armijo-type}~\mbox{step-size}~\mbox{criterion}: 
\begin{itemize}[noitemsep,topsep=2pt,leftmargin=!,labelwidth=\widthof{$\bullet$}]
    \item[$\bullet$] \emph{Strategy 1:}\hypertarget{Strategy 1}{} \hspace{-0.1mm}First, \hspace{-0.1mm}we \hspace{-0.1mm}run \hspace{-0.1mm}Algorithm~\hspace{-0.1mm}\ref{alg:gradflow-disc} \hspace{-0.1mm}with \hspace{-0.1mm}$u_{h}^0\hspace{-0.15em}=\hspace{-0.15em}0\hspace{-0.15em}\in\hspace{-0.15em}  V_h$ \hspace{-0.1mm}until \hspace{-0.1mm}${\smash{\|\mathtt{F}_{h}^{\varepsilon}(z_h^{k_0+1},u_h^{k_0+1})\|_{Y_h\hspace{-0.1em}\times\hspace{-0.1em} V_h}}\hspace{-0.2em}<\hspace{-0.15em}\frac{1}{4}}$,~\hspace{-0.1mm}where $z_{h}^{k_0+1}\hspace{-0.1em}\coloneqq \hspace{-0.1em}\phi_{\varepsilon}'(|\nabla u^{k_0}_{h}|)|\nabla u^{k_0}_{h}|^{-1}\nabla u^{k_0+1}_{h}\hspace{-0.1em}\in\hspace{-0.1em} Y_h$,
    for some $k_0\in\mathbb{N}$. Second, we run~\mbox{Algorithm}~\ref{alg:SSNM-disc}~with $(z_{h}^0,u_{h}^0)\hspace{-0.1em}\coloneqq \hspace{-0.1em}(z_{h}^{k_0+1},u_{h}^{k_0+1})\hspace{-0.1em}\in \hspace{-0.1em} Y_{
h}\times V_{h}$ until $\smash{\|\mathtt{F}_{h}^{\varepsilon}(z_h^{k^*+1},u_h^{k^*+1})\|_{Y_h\hspace{-0.1em}\times \hspace{-0.1em}V_h}}\hspace{-0.1em}<\hspace{-0.1em}1\mathrm{e}{-}{12}$~for~some~${k^*\hspace{-0.1em}\in \hspace{-0.1em}\mathbb{N}_{\leq 250}}$; 

    \item[$\bullet$] \emph{Strategy 2:}\hypertarget{Strategy 2}{} We run Algorithm~\ref{alg:SSNM-disc} with $(z_{h}^0,u_{h}^0)=(\mathtt{0}_d,0)\in Y_h\times V_h$ and Armijo backtracking~line search until $\smash{\|\mathtt{F}_{h}^{\varepsilon}(z_h^{k^*+1},u_h^{k^*+1})\|_{Y_h\times V_h}}<1\mathrm{e}{-}{12}$ for some $k^*\in \{0,\ldots,250\}$.
\end{itemize}

For both strategies in Figure \ref{fig:test1}, we observe that the iteration count increases for decreasing~values of $\varepsilon$. In total, Strategy \hyperlink{Strategy 1}{1} requires about four times more iterations (counted as $k_0 +k^*$
 and $k^*$, respectively) than Strategy \hyperlink{Strategy 2}{2} to satisfy the prescribed stopping criteria~in~the~\mbox{employed}~\mbox{algorithms}. For~$\beta\in\{2.5,3\}$, no convergence was observed within 250 iterations when using Strategy \hyperlink{Strategy 2}{2}.~In~all cases, the $\operatorname{prox}$-based semi-smooth Newton method (\textit{cf}.\ Algorithm~\ref{alg:SSNM-disc}) reduces the discrete primal-dual~gap estimator $\eta_{\mathtt{gap},h}^2$ (\textit{cf}.\ Theorem~\ref{thm:discrete_primal-dual_gap_identity}\eqref{thm:discrete_primal-dual_gap_identity.2.3}), which controls the discrete primal-dual distance between primal-dual iterates and the exact discrete primal-dual solution,~to~machine~\mbox{precision}.

\begin{figure}[H]
    \centering
    \includegraphics[width=\linewidth]{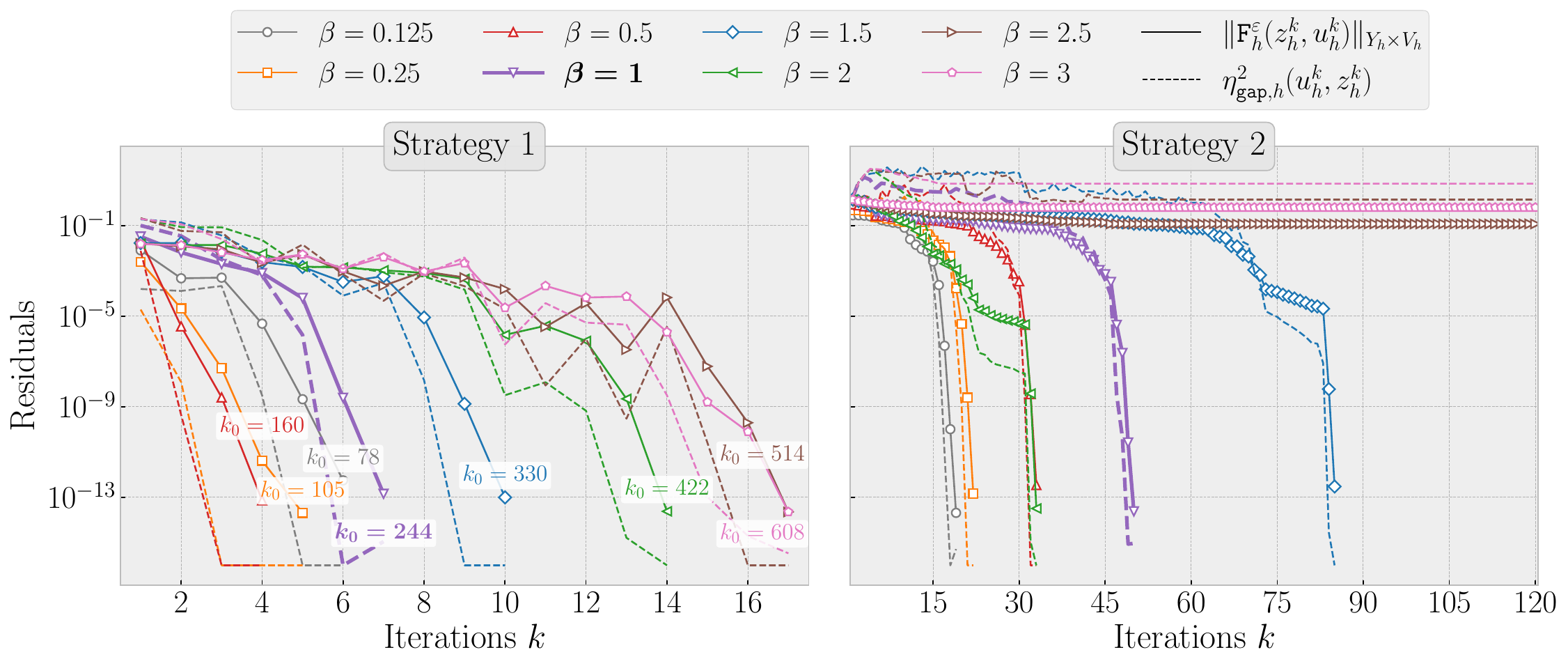}\vspace{-2mm}
    \caption{Residual decay of the $\operatorname{prox}$-based semi-smooth Newton method (\textit{cf}.\ Algorithm~\ref{alg:SSNM-disc}) for different choices of the regularization parameter $\varepsilon = h_7^\beta$ with $\beta \in \{0.125,0.25,0.5,1,1.5,2,2.5,3\}$. Left: Strategy~\protect\hyperlink{Strategy 1}{1} combining Algorithm~\ref{alg:gradflow-disc} with Algorithm~\ref{alg:SSNM-disc}; only the iterations of Algorithm~\ref{alg:SSNM-disc} are displayed, with $k_0\in \mathbb{N}$
  indicating the number of preceding iterations of Algorithm~\ref{alg:gradflow-disc}. Right: Strategy~\protect\hyperlink{Strategy 2}{2} using Algorithm~\ref{alg:SSNM-disc} with Armijo backtracking line search. For decreasing~values~of~$\varepsilon$, iteration counts increase, whereas Strategy~\protect\hyperlink{Strategy 1}{1} is more robust at the expense of additional iterations.}
    \label{fig:test1}
\end{figure}\vspace{-5mm}

\subsection*{Test 2: Role of proximity parameter $\gamma$}\enlargethispage{6.5mm}\vspace{-0.5mm}

\hspace{5mm}In this test, again, we restrict to the triangulation $\mathcal{T}_{h_7}$, fix the regularization~\mbox{parameter}~$\varepsilon = h$, and consider different choices of the proximity parameter $\gamma = 2^{\ell}$ for $\ell \in \{-2,-1,0,1,2,\ldots,9\}$. The dual invariance in Section~\ref{subsec:dual-invariance} suggests that $\gamma =\mathcal{O}(1)$ generally~provides~a~\mbox{robust}~choice.\linebreak This hypothesis is supported by the experimental results reported in Figure~\ref{fig:test2}: Although the choice $\gamma = 1$ is not associated with the minimal iteration count in every case, excessively small or large values of $\gamma$ lead either to a significant increase in the iteration count or to~failure~(\textit{e.g.},~stagnation) of the $\operatorname{prox}$-based semi-smooth Newton method (\textit{cf}. Algorithm~\ref{alg:SSNM-disc}).\vspace{-1mm} 

\begin{figure}[H]
    \centering
    \includegraphics[width=\linewidth]{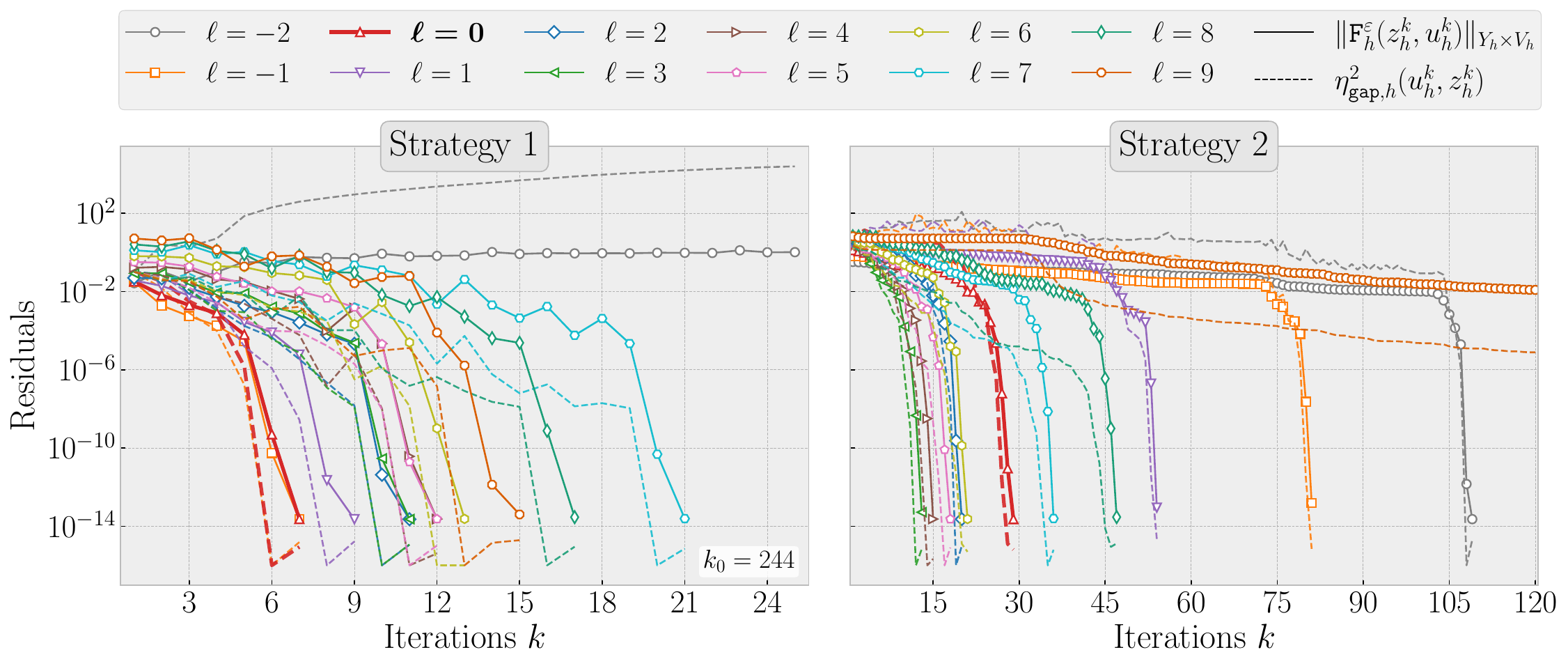}\vspace{-2mm}
    \caption{Residual decay of the $\operatorname{prox}$-based semi-smooth Newton method (\textit{cf}.~Algorithm~\ref{alg:SSNM-disc}) for different choices of the proximity parameter $\gamma = 2^\ell$ with $\ell \in \{-2,-1,0,1,\dots,9\}$.~For~Strategy~\protect\hyperlink{Strategy 1}{1}, only the iterations of Algorithm~\ref{alg:SSNM-disc} are displayed, with $k_0\in \mathbb{N}$ indicating the number of preceding iterations of Algorithm \ref{alg:gradflow-disc}, which is identical for all $\ell \in \{-2,-1,0,1,\dots,9\}$. While~${\gamma =1}$~does~not\linebreak always yield the smallest iteration count, it appears to provide the most robust convergence~beha\-vior overall, whereas,~for~\mbox{excessively}~small~or~large~\mbox{values}~of~$\gamma$,~the~\mbox{iteration}~may~fail~to~terminate.}
    \label{fig:test2}
\end{figure}

\subsection*{Test 3: Comparison to canonical primal semi-smooth Newton method}\enlargethispage{7.5mm}\vspace{-0.5mm}

\hspace{5mm}In this test, we compare the $\operatorname{prox}$-based semi-smooth Newton method (\textit{cf}.\ Algorithm~\ref{alg:SSNM-disc}) with a canonical primal semi-smooth Newton method applied directly to the Fr\'echet derivative $\mathrm{D}I_h^{\varepsilon}\colon V_h\to V^*_h$ of 
the discrete primal functional \eqref{eq:discrete_primal}, which is globally Newton differentiable with possible (global) Newton derivative $\mathtt{J}_{\smash{\mathrm{D}I_h^{\varepsilon}}}\colon  V_h\to  \mathcal{L}(V_h;V^*_h)$,~for~every~${u_h,v_h,w_h\in  V_h}$~\mbox{defined}~by\vspace{-0.75mm}
\begin{align*}
\langle\mathtt{J}_{\smash{\mathrm{D}I_h^{\varepsilon}}}(u_h) w_h,
v_h
\rangle_{V_h}
\coloneqq
\int_\Omega
\mathtt{J}_{\mathrm{D}\phi_\varepsilon}(\nabla u_h)
\nabla w_h\cdot \nabla v_h
\,\mathrm{d}x
+
\alpha\int_\Omega w_hv_h\,\mathrm{d}x\,,\\[-6.25mm]\notag
\end{align*}
 where $\mathtt{J}_{\mathrm{D}\phi_\varepsilon}\colon \mathbb{R}^d\to \mathbb{R}^{d\times d}$, for every $t\in \mathbb{R}^d$ defined by\vspace{-1.25mm}
 \begin{align*}
    \mathtt{J}_{\mathrm{D}\phi_\varepsilon}(t)
\coloneqq
\begin{cases}
    \frac{1}{\varepsilon}\mathbbone &\text{ if }\vert t\vert<\varepsilon\,,\\
    \frac{1}{\vert t\vert}\mathtt{P}_t&\text{ if }\vert t\vert\ge \varepsilon\,,
\end{cases}\\[-6.5mm]\notag
 \end{align*}
 denotes a possible (global) Newton derivative of $\mathrm{D}\phi_\varepsilon\colon \mathbb{R}^d\to \mathbb{R}^d$.

\begin{algorithm}[Primal semi-smooth  Newton method for discrete, regularized~\mbox{TV-minimization}]
\label{alg:Newton-disc}
Let $\varepsilon\hspace{-0.1em}>\hspace{-0.1em}0$ be a regularization parameter, let
$\varepsilon^{\mathtt{stop}}_{\mathtt{loc},h}\hspace{-0.1em}>\hspace{-0.1em}0$ be a stopping parameter,~let~$u_h^0\hspace{-0.1em}\in \hspace{-0.1em}V_h$~be~an~ini\-tial iterate, and let $k^{\mathtt{max}}_{\mathtt{loc},h}\hspace{-0.15em}\in\hspace{-0.15em}\mathbb{N}\cup\{+\infty\}$ be a maximal 
number~of~iterations.~Then,~for~${k\hspace{-0.1em}=\hspace{-0.1em}0,\ldots,\smash{k^{\mathtt{max}}_{\mathtt{loc},h}}}$, perform the following iteration loop:
\begin{itemize}[noitemsep,topsep=2pt,leftmargin=!,labelwidth=\widthof{(2)}]
\item[(1)] \hypertarget{alg:Newton-disc.1}{} Compute the Newton update direction $\delta \smash{u_h^k\in V_h}$ such that\vspace{-0.5mm}
\begin{align*}
    \mathtt{J}_{\smash{\mathrm{D}I_h^{\varepsilon}}}(u_h^k)\delta u_h^k = -\mathrm{D}I_h^{\varepsilon}(u_h^k)
\quad\text{ in }(V_h)^*\,,\\[-6mm]\notag
\end{align*} 
and the updated iterate $u_h^{k+1}\coloneqq u_h^k+\delta u_h^k\in V_h$; 

\item[(2)]  \hypertarget{alg:Newton-disc.2}{} If $\smash{\|\mathrm{D}I_h^{\varepsilon}(u_h^{k+1})\|_{V_h^\ast}<\smash{\varepsilon^{\mathtt{stop}}_{\mathtt{loc},h}}}$, 
then \textup{STOP}; otherwise, set $k\to k+1$ and continue with~(\hyperlink{alg:Newton-disc.1}{1}).
\end{itemize}
\end{algorithm}

In this test, we consider the sequence of triangulations
$\{\mathcal{T}_{h_i}\}_{i=0,\ldots,7}$ constructed above,~the regularization \hspace{-0.1mm}parameters \hspace{-0.1mm}$\varepsilon\hspace{-0.1em}\in\hspace{-0.1em}\{h,h^2\}$,
\hspace{-0.1mm}and \hspace{-0.1mm}fix \hspace{-0.1mm}the \hspace{-0.1mm}proximity \hspace{-0.1mm}parameter \hspace{-0.1mm}$\gamma\hspace{-0.1em}=\hspace{-0.1em}1$. \hspace{-0.1mm}For ~\hspace{-0.1mm}both~\hspace{-0.1mm}the~\hspace{-0.1mm}\mbox{$\operatorname{prox}$-based} semi-smooth Newton~method (\textit{cf}.\ Algorithm~\ref{alg:SSNM-disc}) and the primal semi-smooth Newton method (\textit{cf}.\ Algorithm~\ref{alg:Newton-disc}), \hspace{-0.1mm}for \hspace{-0.1mm}$i=0,\ldots,7$,
\hspace{-0.1mm}we \hspace{-0.1mm}use \hspace{-0.1mm}the \hspace{-0.1mm}last \hspace{-0.1mm}iterate \hspace{-0.1mm}of \hspace{-0.1mm}Algorithm~\hspace{-0.1mm}\ref{alg:gradflow-disc},
\hspace{-0.1mm}initialized~\hspace{-0.1mm}with~\hspace{-0.1mm}$ {u_h^0=0\in V_h}$ \hspace{-0.15mm}and \hspace{-0.15mm}terminated \hspace{-0.15mm}once
\hspace{-0.15mm}$\|\smash{\mathtt{F}_h^\varepsilon(z_h^{k_0+1},u_h^{k_0+1})}\|_{Y_h\hspace{-0.1em}\times\hspace{-0.1em} V_h}\hspace{-0.175em}<\hspace{-0.175em}\frac{1}{5}$, \hspace{-0.15mm}where \hspace{-0.15mm}${z_{h}^{k_0+1}\hspace{-0.175em}\coloneqq \hspace{-0.175em}\phi_{\varepsilon}'(|\nabla u^{k_0}_{h}|)|\nabla u^{k_0}_{h}|^{-1}\nabla u^{k_0+1}_{h}\hspace{-0.175em}\in\hspace{-0.175em} Y_h}$,
    for some $k_0\in\mathbb{N}$, 
    as initial iterate. In order to ensure comparability, we monitor the residual decay in terms of the primal-dual gap estimator
$\eta_{\mathtt{gap},h}^2$
(\textit{cf}.\ Theorem~\ref{thm:discrete_primal-dual_gap_identity}\eqref{thm:discrete_primal-dual_gap_identity.2.3}),
reported~in~\mbox{Figure}~\ref{fig:test3}. The results indicate that the convergence basin of the primal semi-smooth Newton method
(\textit{cf}.\ Algorithm~\ref{alg:Newton-disc})
appears more sensitive with respect to mesh-refinement.\vspace{-1.5mm}

\begin{figure}[H]
    \centering
    \includegraphics[width=\linewidth]{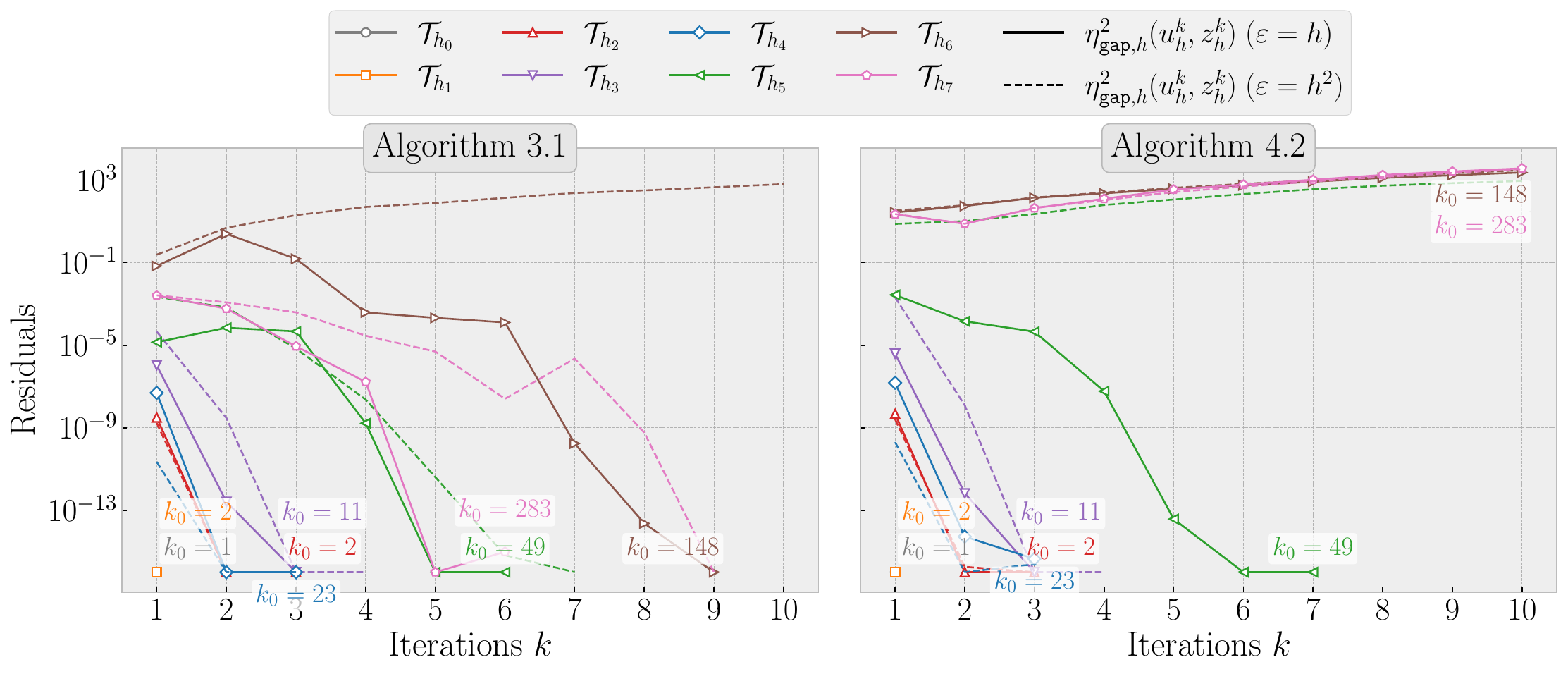}\vspace{-2mm}
    \caption{Residual decay (in terms of $\eta_{\mathtt{gap},h}^2(u_h^k,z_h^k)$, $k\in \mathbb{N}_0$, \textit{cf}.\ \eqref{thm:discrete_primal-dual_gap_identity.2.3})
of the $\operatorname{prox}$-based semi-smooth Newton method
(\textit{cf}.\ Algorithm \ref{alg:SSNM-disc})
and the primal semi-smooth~\mbox{Newton}~method~(\textit{cf}.~Algorithm \hspace{-0.1mm}\ref{alg:Newton-disc})
\hspace{-0.1mm}for \hspace{-0.1mm}the \hspace{-0.1mm}sequence \hspace{-0.1mm}of \hspace{-0.1mm}triangulations~\hspace{-0.1mm}$\{\mathcal{T}_{h_i}\}_{i=0,\ldots,7}$~\hspace{-0.1mm}and~\hspace{-0.1mm}\mbox{regularization}~\hspace{-0.1mm}\mbox{parameters}~\hspace{-0.1mm}${\varepsilon\hspace{-0.175em}\in\hspace{-0.175em}\{h,h^2\}}$.
The results indicate that the convergence basin of the $\operatorname{prox}$-based semi-smooth Newton method (\textit{cf}.\ Algorithm \ref{alg:SSNM-disc}) is more robust under mesh-refinement.}
    \label{fig:test3}
\end{figure}

\subsection*{Test 4: Locally refined triangulations}

\hspace{5mm}In this test, we choose the experimental setup from \cite[Subsec.\ 6.3]{BartelsToveyWassmer2022}. More precisely,~we~consi\-der a sequence of triangulations $\{\mathcal{T}_{h_i}\}_{i\in \mathbb{N}}$, by
means of red-green-blue refinement~\mbox{starting} from an initial triangulation 
$\mathcal{T}_0$ constructed as above and iteratively marking~all~triangles~that~have a \hspace{-0.1mm}non-empty \hspace{-0.1mm}intersection \hspace{-0.1mm}with \hspace{-0.1mm}the \hspace{-0.1mm}discontinuity \hspace{-0.1mm}set \hspace{-0.1mm}$J_u\hspace{-0.1em} =\hspace{-0.1em} \partial B_r^2(0)$ \hspace{-0.1mm}of \hspace{-0.1mm}$u\hspace{-0.15em}\in\hspace{-0.15em} BV(\Omega)\cap L^2(\Omega)$~\hspace{-0.1mm}(\textit{cf}.~\hspace{-0.1mm}\mbox{Figure}~\hspace{-0.1mm}\ref{fig:test4_triang}). This leads to an (asymptotically) quadratic grading strength (\textit{cf}.\ \cite[Sec.\ 4]{BartelsToveyWassmer2022}),~\textit{i.e.},~for~every~${i\in \mathbb{N}_0}$, denoting by $h_{\min,i}\hspace{-0.1em}\coloneqq\hspace{-0.1em} \min_{T\in \mathcal{T}_i}{\{h_T\}}$ and $h_{\textup{avg},i}\hspace{-0.1em}\coloneqq \hspace{-0.1em}\vert \mathcal{T}_{h_i}\vert^{-\frac{1}{2}}$ the minimal~and~the average mesh-size, respectively, for the grading exponent $\beta_i=
\log(h_{\min,i})/\log(h_{\textup{avg},i})$, there holds $\beta_i\to 2$ $(i\to \infty)$.
This (asymptotically) quadratic grading strength is experimentally confirmed by~\mbox{Figure}~\ref{fig:test4_triang}.\vspace{-2mm}
\begin{figure}[H]
    \centering
    \includegraphics[width=\linewidth]{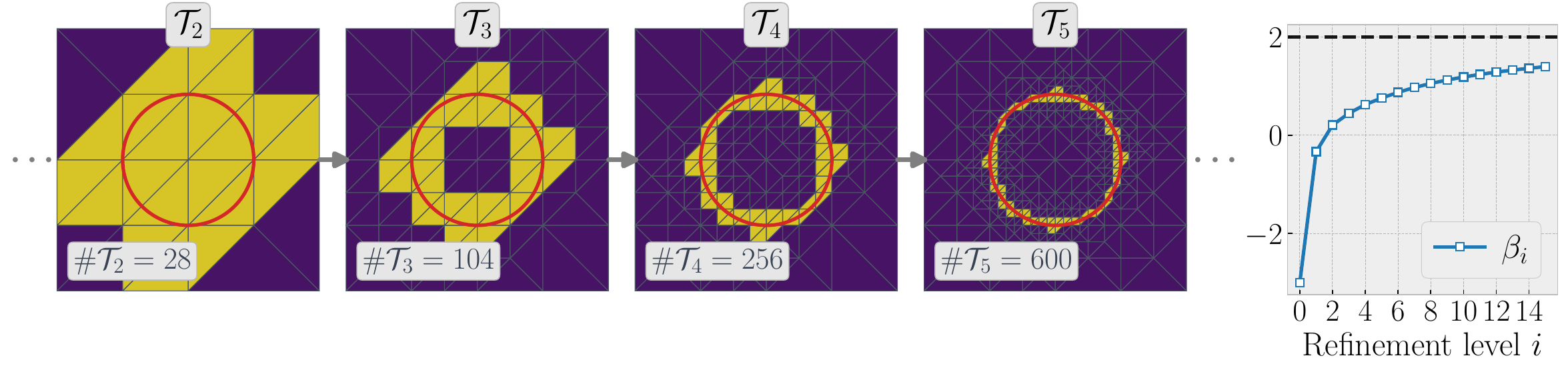}\vspace{-2mm}
    \caption{Locally refined triangulations $\mathcal{T}_i$,
$i \in \{0,2,4,6\}$, obtained by red-green-blue refinement by marking all triangles
intersecting the discontinuity set $J_u=\partial B_r^2(0)$~(marked~in~\mbox{\textcolor{yellow!90!black}{yellow}};~the interface is marked
in \textcolor{red}{red}). 
The plot on the right shows the grading exponents~$\beta_i$,~${i=0,\dots,15}$,~(\textcolor{denim}{blue}). 
The observed convergence $\beta_i\nearrow 2$ $(i\to \infty)$ confirms asymptotically~quadratic~grading.}\vspace{-2mm}
    \label{fig:test4_triang}
\end{figure}

 Given the  experimentally confirmed (asymptotic) quadratic grading strength in Figure~\ref{fig:test4_triang}, in order to preserve the theoretically best possible error decay rate, it is necessary to choose $\varepsilon = h^{2}$.\enlargethispage{1.5mm}
In this test, we replace Strategies
\protect\hyperlink{Strategy 1}{1} and
\protect\hyperlink{Strategy 2}{2}
by modified Strategies
1$^*$
and
2$^*$,
where the initial iterate is generated via
Algorithm~\ref{alg:gradflow-disc}
as in Strategy~\protect\hyperlink{Strategy 1}{1},
but terminated once
${\|\mathtt{F}_h^\varepsilon(z_h^{k_0+1},u_h^{k_0+1})\|_{Y_h\times V_h}\hspace{-0.1em}<\hspace{-0.1em}\frac{1}{10}}$ for some $k_0\in \mathbb{N}$.
Using this initial iterate,
Strategy~1$^*$~employs~Algorithm~\ref{alg:SSNM-disc}~without~line~search, while
Strategy~2$^*$
uses Armijo backtracking line search.
Figure~\ref{fig:test4} indicates~that~\mbox{Strategy}~2$^*$~is~more\linebreak robust. Nevertheless, despite the mesh-independent initialization criterion, both strategies remain effective even for quadratic grading, although iteration~counts~\mbox{increase}~\mbox{under}~\mbox{mesh-refinement}.\enlargethispage{5mm}\vspace{-1mm}

\begin{figure}[H]
    \centering
    \includegraphics[width=\linewidth]{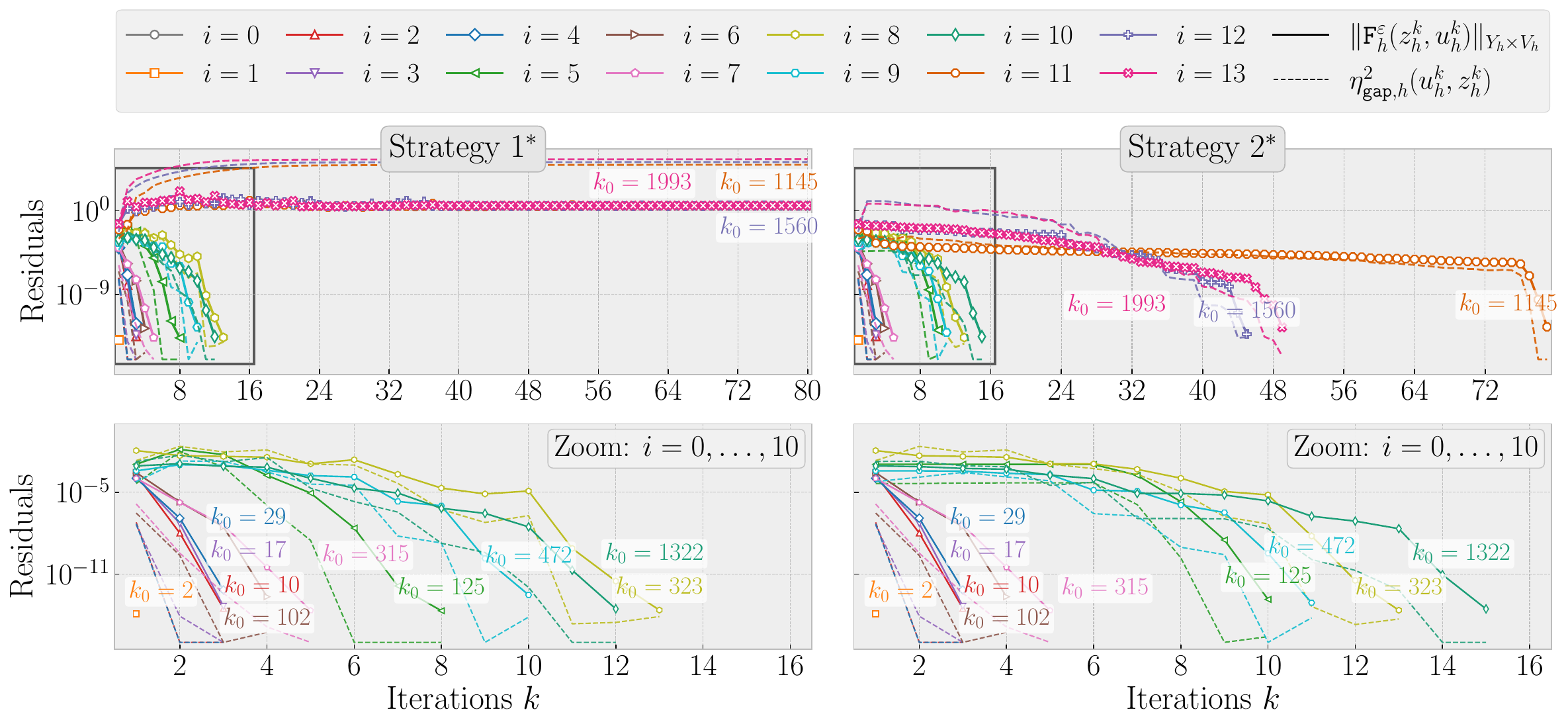}\vspace{-2mm}
    \caption{Residual decay of the $\operatorname{prox}$-based semi-smooth Newton method (\textit{cf}.\ Algorithm~\ref{alg:SSNM-disc}) under mesh-refinement for the modified Strategies
1$^*$
and
2$^*$,
using the same initial~iterate~generated via Algorithm~\ref{alg:gradflow-disc}. Strategy
2$^*$
uses Armijo backtracking line search, whereas Strategy~1$^*$~does~not. Strategy 2$^*$
appears more robust, although iteration counts also increase under mesh-refinement.}
    \label{fig:test4}
\end{figure}

    \newpage
    
	{\setlength{\bibsep}{0pt plus 0.0ex}\small\enlargethispage{5mm}
		\bibliographystyle{aomplain}
		\bibliography{references} 
	}

    \appendix

\section*{Appendix}\vspace{-1mm}

\hspace{5mm}In this appendix, we provide the proof of Theorem \ref{thm:discrete_primal-dual_gap_identity}\eqref{thm:discrete_primal-dual_gap_identity.2}.\vspace{-1mm} 

    \begin{proof}[Proof (of Theorem \ref{thm:discrete_primal-dual_gap_identity}\eqref{thm:discrete_primal-dual_gap_identity.2}).]\let\qed\relax
        \emph{ad \eqref{thm:discrete_primal-dual_gap_identity.2.1}.}
        For every $v_h\in V_h$, using the binomial formula and the discrete convex optimality conditions in \eqref{eq:discrete_optimality}, we find that\vspace{-0.5mm}
        \begin{align*}
            \rho_{I_h^{\varepsilon}}^2(v_h)
            &
            =\int_{\Omega}{\vert \nabla v_h\vert_{\varepsilon}\,\mathrm{d}x}-\int_{\Omega}{\vert \nabla u_h^{\varepsilon}\vert_{\varepsilon}\,\mathrm{d}x}+\frac{\alpha}{2}\int_{\Omega}{(v_h+u_h^{\varepsilon}-2g_h)(v_h-u_h^{\varepsilon})\,\mathrm{d}x}
            \\[-0.25mm]&=\int_{\Omega}{\vert \nabla v_h\vert_{\varepsilon}\,\mathrm{d}x}+\int_{\Omega}{u_h^{\varepsilon}\operatorname{div}_h z_h^{\varepsilon}\,\mathrm{d}x}+\frac{\varepsilon}{2}\int_{\Omega}{\vert z_h^\varepsilon\vert^2\,\mathrm{d}x}
            \\[-0.25mm]&\quad+\int_{\Omega}{\operatorname{div}_h z_h^{\varepsilon}(v_h-u_h^{\varepsilon})\,\mathrm{d}x}+\frac{\alpha}{2}\int_{\Omega}{\vert v_h-u_h^{\varepsilon}\vert^2\,\mathrm{d}x}
            \\[-0.25mm]&=\int_{\Omega}{\vert \nabla v_h\vert_{\varepsilon}\,\mathrm{d}x}+\int_{\Omega}{v_h\operatorname{div}_h z^{\varepsilon}_h\,\mathrm{d}x}+\frac{\varepsilon}{2}\int_{\Omega}{\vert z_h^\varepsilon\vert^2\,\mathrm{d}x}+\frac{\alpha}{2}\int_{\Omega}{\vert v_h-u^{\varepsilon}_h\vert^2\,\mathrm{d}x}\,.\\[-6mm]\notag
        \end{align*}

         \emph{ad \eqref{thm:discrete_primal-dual_gap_identity.2.2}.}
         For every $y_h\in Y_h$ with $\vert y_h\vert \leq 1$ a.e.\ in $\Omega$, using the binomial formula and  discrete convex optimality conditions in \eqref{eq:discrete_optimality}, we find that\vspace{-0.5mm}
         \begin{align*}
             \rho_{-D_h^{\varepsilon}}^2(y_h)
             &
             =\frac{\varepsilon}{2}\int_{\Omega}{\vert y_h\vert^2\,\mathrm{d}x}-\frac{\varepsilon}{2}\int_{\Omega}{\vert z_h^{\varepsilon}\vert^2\,\mathrm{d}x}
                +\frac{1}{2\alpha}\int_{\Omega}{(\operatorname{div}_h(y_h+z_h^{\varepsilon})+2\alpha g_h) \operatorname{div}_h (y_h-z_h^{\varepsilon})\,\mathrm{d}x}
             \\[-0.25mm]&=\frac{\varepsilon}{2}\int_{\Omega}{\vert y_h\vert^2\,\mathrm{d}x}+\int_{\Omega}{\vert \nabla u_h^{\varepsilon}\vert_{\varepsilon}\,\mathrm{d}x}+\int_{\Omega}{u_h^{\varepsilon}\operatorname{div}_h z_h^{\varepsilon}\,\mathrm{d}x}
             \\[-0.25mm]&\quad+\frac{1}{\alpha}\int_{\Omega}{(\operatorname{div}_h z_h^{\varepsilon}+\alpha g_h)\operatorname{div}_h(y_h-z_h^{\varepsilon})\,\mathrm{d}x}+\frac{1}{2\alpha}\int_{\Omega}{\vert\operatorname{div}_h(y_h-z_h^{\varepsilon})\vert^2\,\mathrm{d}x}
             \\[-0.25mm]&=\int_{\Omega}{\vert \nabla u_h^{\varepsilon}\vert_{\varepsilon}\,\mathrm{d}x}+\int_{\Omega}{u_h^{\varepsilon}\operatorname{div}_h y_h\,\mathrm{d}x}+\frac{\varepsilon}{2}\int_{\Omega}{\vert y_h\vert^2\,\mathrm{d}x}+\frac{1}{2\alpha}\int_{\Omega}{\vert\!\operatorname{div}_h( y_h- z_h^{\varepsilon})\vert^2\,\mathrm{d}x}\,.\\[-6mm]\notag
         \end{align*}

         \emph{ad \eqref{thm:discrete_primal-dual_gap_identity.2.3}.}  For every $v_h\in V_h$ and $y_h\in Y_h$ with $\vert y_h\vert \leq 1$ a.e.\ in $\Omega$, using twice a binomial formula, we find that\vspace{-0.5mm}
        \begin{align*}
            \eta_{\mathtt{gap},h}^2(v_h,y_h) 
            &
            =\int_{\Omega}{\vert \nabla v_h\vert_{\varepsilon}\,\mathrm{d}x}+\int_{\Omega}{v_h\operatorname{div}_h y_h\,\mathrm{d}x}+\frac{\varepsilon}{2}\int_{\Omega}{\vert y_h\vert^2\,\mathrm{d}x}\\[-0.25mm]&\quad+\frac{1}{2\alpha}\int_{\Omega}{\vert\!\operatorname{div}_h y_h\vert^2\,\mathrm{d}x}
            -\int_{\Omega}{\operatorname{div}_h y_h(v_h-g_h)\,\mathrm{d}x}
            +\frac{\alpha}{2}\int_{\Omega}{\vert v_h-g_h\vert^2\,\mathrm{d}x}
            \\[-0.25mm]&=\int_{\Omega}{\vert \nabla v_h\vert_{\varepsilon}\,\mathrm{d}x}+\int_{\Omega}{v_h\operatorname{div}_h y_h\,\mathrm{d}x}+\frac{\varepsilon}{2}\int_{\Omega}{\vert y_h\vert^2\,\mathrm{d}x}\\[-0.25mm]&\quad+\frac{1}{2\alpha}\int_{\Omega}{\vert\!\operatorname{div}_h y_h-\alpha(v_h-g_h)\vert^2\,\mathrm{d}x}\,.\tag*{$\qedsymbol$}\\[-6mm]\notag
        \end{align*}
         
    \end{proof}

\end{document}